\documentclass[11pt,oneside,reqno]{amsart}
\usepackage{geometry}
\usepackage[draft, textwidth=1.2in]{todonotes}
\usepackage{color}\usepackage[centertags]{amsmath}
\usepackage{amsfonts}
\usepackage{amssymb}
\usepackage{amsthm,color}
\usepackage{newlfont}
\usepackage{bbm}

\usepackage{mathtools}

\usepackage{hyperref}
\hypersetup{
  colorlinks   = true, 
  urlcolor     = blue, 
  linkcolor    = red, 
  citecolor   = red 
}
\usepackage{cleveref}
\theoremstyle{definition}
\usepackage{graphicx}

\setkeys{Gin}{width=\linewidth,totalheight=\textheight,keepaspectratio}
\graphicspath{{./images/}}

\usepackage{multicol}
\usepackage{booktabs}
\usepackage{mathrsfs}\usepackage{color}

\newcommand{\nbb}{\mathbb{N}}
\newcommand{\rbb}{\mathbb{R}}

\newcommand{\C}{\mathcal{C}}
\renewcommand{\S}{\mathcal{S}}

\renewcommand{\H}{\mathcal{H}}

\newcommand{\K}{\mathcal{K}}

\newcommand{\Atilde}{\tilde{A}}

\newcommand{\xbar}{\overline{x}}
\newcommand{\vbar}{\overline{v}}

\newcommand{\xibar}{\overline{\xi}}
\newcommand{\xtil}{\widetilde{x}}
\newcommand{\xhat}{\widehat{x}}
\newcommand{\vtil}{\widetilde{v}}
\newcommand{\vhat}{\widehat{v}}

\newcommand{\Ktilde}{\widetilde{K}}
\newcommand{\xitilde}{\tilde{\xi}}
\newcommand{\Ftilde}{\tilde{F}}
\newcommand{\Wtilde}{\tilde{W}}
\newcommand{\xihat}{\widehat{\xi}}
\newcommand{\sigmatilde}{\widetilde{\sigma}}
\newcommand{\U}{U}

\newcommand{\Cpastrho}{C_\varrho((-\infty,0],\mathbb{R}^2)}

\newcommand{\f}{\varphi}

\newcommand{\E}{\mathbb{E}}

\renewcommand{\P}{\mathbb{P}}

\newcommand{\Hs}{\mathcal{H}_{-s}}
\newcommand{\Z}{\mathbf{Z}}

\newcommand{\close}{\!\!\!}
\newcommand{\Law}{\mathrm{Law}}
\newcommand{\ccdot}{\;\cdot\;}

\theoremstyle{definition}
\theoremstyle{plain}
\newtheorem{theorem}{Theorem}[section]
\newtheorem{corollary}[theorem]{Corollary}

\newtheorem{lemma}[theorem]{Lemma}
\newtheorem{assumption}[theorem]{Assumption}
\newtheorem{proposition}[theorem]{Proposition}

\theoremstyle{definition}
\newtheorem{definition}[theorem]{Definition}

\newtheorem{remark}[theorem]{Remark}

\numberwithin{equation}{section}

\crefname{algorithm}{Algorithm}{Algorithms}
\crefname{assumption}{Assumption}{Assumptions}
\crefname{lemma}{Lemma}{Lemmas}
\crefname{theorem}{Theorem}{Theorems}
\crefname{remark}{Remark}{Remarks}
\crefname{corollary}{Corollary}{Corollaries}
\crefname{figure}{Fig.}{Figures}
\crefname{section}{Section}{Sections}
\crefname{proposition}{Proposition}{Propositions}
\crefname{definition}{Definition}{Definitions}

\author[ D.~Herzog and H.~Nguyen]{David P.~Herzog$^1$, Jonathan C. Mattingly$^2$ and Hung D.~Nguyen$^3$}

\address{$^1$ Department of Mathematics, Iowa State University, Iowa, USA}
\address{$^2$ Department of Mathematics, Duke University, North Carolina, USA}
\address{$^3$ Department of Mathematics, University of California, Los Angeles, California, USA}

\begin{document}

\title{Gibbsian dynamics and the generalized Langevin equation}
\maketitle

\begin{abstract}
We study the statistically invariant structures of the nonlinear generalized Langevin equation (GLE) with a power-law memory kernel. For a broad class of memory kernels, including those in the subdiffusive regime, we construct solutions of the GLE using a Gibbsian framework, which does not rely on existing Markovian approximations. Moreover, we provide conditions on the decay of the memory to ensure  uniqueness of statistically steady states, generalizing previous known results for the GLE under particular kernels as a sum of exponentials.  
\end{abstract}

\section{Introduction}

\subsection{Overview} \label{sub:overview}
We study the generalized Langevin equation 
\begin{align}
d\,x(t)&=v(t)\,dt,\notag\\
d\,v(t)&=-v(t)\, dt- \U'(x(t)) \, dt- \int_{-\infty}^t\close K(t-r)v(r)\, dr \, dt +  \sqrt{2} \,dW(t)+ F(t) \, dt,\label{eqn:GLE}
\end{align}
describing the motion of a particle with position $x(t)\in \rbb$ and velocity $v(t)\in \rbb$ in a potential~$\U$.  The particle is subject to a viscous friction force $- v(t)\,dt$ and a convolution term involving the \emph{convolution kernel} $K$, modeling a thermal drag force with memory effects.  By the \emph{fluctuation-dissipation} relation, both of these forces are respectively balanced by stochastic processes $W(t)$ and $F(t)$, where $W(t)$ is a standard one-dimensional Brownian motion and $F(t)$ is a mean-zero stationary Gaussian process with covariance given by 
\begin{align}\label{eq:cov}
\E[F(t_1)F(t_2)]=K(|t_1-t_2|), \,\,\,\text{ for all } t_1, t_2 \in \rbb.  
\end{align}
Note that the memory in equation~\eqref{eqn:GLE} is present both in the integral term with
the kernel $K$ and in the Gaussian process $F(t)$ which is not white
in time.

In the absence of memory effects, that is setting $K\equiv 0$ and $F\equiv 0$ in~\eqref{eqn:GLE} above, large-time properties of the resulting Markovian system are well-understood, in the sense that under general conditions on the potential $\U$, it is known
 that the system admits a unique ergodic invariant measure $\pi(x,v)$
 on $\mathbb{R}^2$ which is exponentially attractive and whose formula is given by
\begin{equation} \label{eqn:marginal-x-v}
	\pi(dx,dv) \propto \exp(- H(x,v) ) \,dx\, dv,
\end{equation}
where 
\begin{align*}
H(x,v) = \frac{ v^2}{2}+ \U(x)
\end{align*}
denotes the Hamiltonian of the system.  For example, see~\cite{conrad2010construction, cooke2017geometric, herzog2019ergodicity,mattingly2002ergodicity, pavliotis2014stochastic, Villani_06} and the references within. When $K\equiv 0$ and $F\equiv
0$, one can equally speak of
stationary solutions in path space $C(\rbb ; \mathbb{R}^2)$ of  \eqref{eqn:GLE}  as they are in one-to-one
correspondence with the invariant measures on $\mathbb{R}^2$, namely the fixed points of
the Markov semigroup generated by \eqref{eqn:GLE} without the memory
terms. Here, a process $X(t)$, $t\in(-\infty,\infty)$, is called \emph{stationary} if the distribution
$$\big(X(t_1+s),\dots,X(t_n+s)\big),\quad t_1<\dots<t_n,$$
does not depend on $s$. For further discussion, see Sections \ref{sec:SkewInvMeasure}--\ref{sec:MarkovInvMeasure} below. On the other hand, in the presence of memory in~\eqref{eqn:GLE},
comparatively much less is known about both the existence and
uniqueness of statistically stationary states under general conditions
on $K$.  The goal of this paper is to make progress on bridging this
gap between the standard Langevin equation ($K\equiv 0, F\equiv 0$)
and its generalized counterpart~\eqref{eqn:GLE} with memory.  

In general, there is no Markovian dynamics on $\mathbb{R}^2$ associated
with~\eqref{eqn:GLE}; and hence, no directly analogous concept of an invariant
measure on $\mathbb{R}^2$. Thus, we are left to study the stationary solutions of
~\eqref{eqn:GLE}   in $C(\rbb ; \mathbb{R}^2)$ as
this concept remains well-defined.
One can always associate such a stationary solution to a deterministic
dynamical system $X=(x,v,W,F)$ where $X \in C(\rbb ; \mathbb{R}^4)$ 
represents the dynamics lifted to the path space. Here 
the dynamics is given by the shift map $\theta_t\colon
C(\rbb;  \mathbb{R}^4) \rightarrow
C(\rbb; \mathbb{R}^4)$ defined by
\begin{equation} \label{form:theta_t}
(\theta_t X)(s)=\theta_tX(s)=X(t+s),
\end{equation}
for $X=(x,v,W,F) \in C(\rbb; \mathbb{R}^4)$. As with any deterministic
dynamical system, we can view this as a (nonrandom) Markov process
whose invariant measures are the stationary measures of
$X=(x,v,W,F)$. However, the phase space of such a representation is so
large to be almost useless. The concepts of ``future'' and ``past''
which are so powerful in a more standard Markovian representation have
little power in this context. In particular, the future trajectories encode the past and
hence do not necessarily have the same strong independence properties
enjoyed by a more standard Markovian structure.  One of the central themes of this note is that there are representations lying between the standard Markov representation of memoryless
Langevin dynamics on $\mathbb{R}^2$ and the lifted dynamics to the
path space $C(\rbb; \mathbb{R}^4)$.  Moreover, these representations can be applied in a fruitful way to the case of the generalized Langevin equation~\eqref{eqn:GLE}.

Although there is no general way to represent solutions of~\eqref{eqn:GLE} as a Markov process on $\mathbb{R}^2$, there are special cases where one can still define a convenient Markov process
associated to~\eqref{eqn:GLE} on an extended state space. In particular, when the memory kernel $K(t)$ can be written as a finite sum of exponentials; that is, 
\begin{align}
\label{eqn:Kexpr}
K(t) = \sum_{k=1}^n c_k e^{-\lambda_k t},
\end{align}
for some constants $c_k, \, \lambda_k >0$, one can augment the resulting system~\eqref{eqn:GLE} by a finite number of auxiliary variables to produce a Markov process on a higher, but finite-dimensional space.  This corresponding finite-dimensional system was studied rigorously in \cite{ottobre2011asymptotic,pavliotis2014stochastic}.  There, under general hypotheses on $\U$, it was shown that the system is uniquely ergodic and the marginal invariant distribution of the pair $(x,v)$ is precisely $\pi$ as in~\eqref{eqn:marginal-x-v} \cite{goychuk2012viscoelastic,ottobre2011asymptotic,
pavliotis2014stochastic}.  However, because the sum above is finite, it cannot describe a kernel with power-law decay, i.e., a kernel $K(t)$ satisfying 
\begin{align}
\label{eqn:Kalpha}
K(t) \sim t^{-\alpha}\,\, \text{ as } \,\,t\rightarrow\infty,
\end{align}
 for some $\alpha >0$. Subsequently, this approach was extended to
 handle such memory kernels by writing $K$ as an infinite-sum of
 exponentials $(n\rightarrow \infty, c_k=c_k(\lambda_k)>0$
 in~\eqref{eqn:Kexpr})~\cite{glatt2018generalized}. See
 Remark~\ref{rem:specialform} below.  The resulting dynamics is an
 infinite-dimensional Markov process on a Sobolev-like space and still has a meaningful sense of ``future'' and
 ``past''.  In particular, the process is amenable to classical Markovian techniques despite being infinite dimensional.

In this infinite-dimensional context~\cite{glatt2018generalized}, it
 was shown that there exists an explicit invariant probability measure
 whose $(x,v)$--marginal agrees with~\eqref{eqn:marginal-x-v}. 
 This is true for memory kernels in this specific form regardless of the memory decay rate $\alpha >0$ as in~\eqref{eqn:Kalpha}.  However, to establish uniqueness of this measure, the restriction $\alpha \in (1, \infty)$ as in~\eqref{eqn:Kexpr} was imposed leaving out the important \emph{subdiffusive regime} of $\alpha \in (0,1)$ (see the discussion in Section~\ref{sub:science} below).  One of our goals here is to push through this threshold.  Additionally, we will study~\eqref{eqn:GLE} for the Gaussian forcing as
   in~\eqref{eq:cov}  both when the memory kernel satisfies the
   structural assumption in \eqref{eqn:Kexpr} with $n=\infty$ and
   alternatively when the memory kernel has power-law
   decay~\eqref{eqn:Kalpha} but cannot be expressed as an infinite sum
   of exponentials.

 For general stationary Gaussian forcing $F$, there is not 
 necessarily a Markovian dynamics associated
 to~\eqref{eqn:GLE}.\footnote{One can always consider as the
   state space the path space
   of a process on the time interval $(-\infty,\infty)$. The
   dynamics is then the deterministic shift of the
   trajectories. Lifting of the deterministic process to pathspace is
   not the type of stochastic Markov dynamics we seek.} Hence, we lack
 a natural notion of an Markov invariant measure and study the stationary
 solutions of ~\eqref{eqn:GLE} instead. We give general conditions
 guaranteeing that there is at most one stationary solution.
Although there is no Markov formulation of the stochastic dynamics,
 there is however a natural \emph{skew-flow} on the \emph{infinite past}
 $C((-\infty, 0]; \rbb^2)$ of the trajectories of $(x,v)$ fibered over
 the Gaussian forcing $F$. That is, given a noise realization and an initial trajectory on $(-\infty,0]$, we evolve~\eqref{eqn:GLE} on $[0,t]$, $0<t$, hence obtaining a solution path on $(-\infty,t]$. See Section~\ref{sec:Skew} for a more detailed discussion.

 When~\eqref{eqn:Kexpr} holds with $n<\infty$ or $n=\infty$, then there is a natural Markovian
 formulation of the stochastic dynamics
 \cite{glatt2018generalized,ottobre2011asymptotic,pavliotis2014stochastic}. We will study a different
 Markovian formulation than used in those works.
 The  assumption
 in \eqref{eqn:Kexpr}  implies that $F(t)$ can be constructed as a
 functional of a
 (possibly) infinite collection of independent Brownian Motions on the
 time interval $(-\infty, t]$. We formulate a
 Markovian dynamics which takes as its state space the trajectories
 of $(x,v)$ on the
 \emph{infinite past} $C((-\infty, 0]; \rbb^2)$ and the infinite past
 of the collection of independent Brownian Motions used to construct
 $F$. We show that when $\alpha>1/2$, this dynamics has at most one invariant measure; or
 equivalently, at most one stationary solution, cf. Theorem~\ref{thm:unique}.

 \begin{remark}\label{rem:Gibbs}   {\bf Gibbsian Dynamics:} As previously mentioned above, it is possible
   to enlarge the state space of any dynamics to make it
   Markovian. In the extreme, by making the state space the entire
   trajectory $\{\big(x(t),v(t),F(t)\big) : t \in(-\infty,\infty)\}$, the dynamics is simply the shift map
   $\theta_t\colon (x,v,F)\mapsto
   \big(x(\ccdot+t),v(\ccdot+t),F(\ccdot+t)\big)$. At this level of
   generality, the fact that the dynamics is Markovian provides little
   useful structure. However, our setting below has more structure.

   In the continuous-time Markov setting, the distribution
   of infinitesimal increments is a function the current state of the
   process. In the Gibbsian setting, as envisioned in
   \cite{weinan2001gibbsian,Mattingly02}, the distribution
   of infinitesimal increments is a function of the entire past.
   We will return to this setting in Section~\ref{sec:simplfiedSetting}.
   The term
   Gibbssian comes from the dynamics being dictated, not by a
   compatible family of Markov measures (depending only on the
   boundary data in space-time), but rather a compatible family of
   Gibbs measures (in the general sense of \cite{Georgii}). 
 \end{remark}

\subsection{Physical motivation}
\label{sub:science}
It is important to note some of the physical reasons for considering memory kernels $K$ in general, and in the power-law regime in particular. The standard Langevin equation is commonly used to describe microparticle motion embedded in Newtonian fluids, which amounts to the implicit assumption that there is no time correlation between the foreign microparticles and the thermally fluctuating fluid molecules. Following Newton's Second Law \cite{pavliotis2014stochastic}, the two-dimensional Langevin equation has the form~\eqref{eqn:GLE} with $K\equiv 0$ and $F\equiv 0$.  On the other hand, for viscoelastic fluids, elasticity induces time correlation between foreign particles and fluid molecules, leading to memory effects.  Thus the standard Langevin equation is not sufficient to describe the motion of the particles suspended in the fluid. In order to capture such phenomena, the generalized Langevin equation~\eqref{eqn:GLE} with general $K$ was introduced in~\cite{kubo1966fluctuation,mori1965continued,mori1965transport} and later popularized in~\cite{mason1995optical}.

It is known that the unconstrained GLE (i.e. $\U\equiv 0$ in~\eqref{eqn:GLE}) exhibits \emph{anomalous diffusion}; that is, the mean-squared displacement $\E x(t)^2$ may not be asymptotically proportional to  $t$ as $t\rightarrow \infty$.   In fact, it was shown in \cite{didier2020asymptotic,mckinley2018anomalous} that when $K\in L^1(\rbb)$, the unconstrained GLE is \emph{asymptotically diffusive}, i.e., $\E x(t)^2\sim t$ as $t\to\infty$. Otherwise, if $K(t)\sim t^{-\alpha}$, $\alpha\in(0,1)$, then the unconstrained GLE is \emph{asymptotically subdiffusive}, i.e. $\E x(t)^2 \sim t^\alpha$ and when $\alpha=1$, there is a transition phase between \emph{diffusion} and \emph{subdiffusion}, i.e., $K(t)\sim t^{-1}$ implies $\E x(t)^2\sim t/\log(t)$ as $t\to\infty$.  For viscoelastic fluids, the subdiffusive regime is observed in experiments~\cite{ghosh:cherstvy:grebenkov:metzler:2016,levine2000one,
meroz:sokolov:2015,metzler:klafter:2000,saxton1994anomalous,saxton1996anomalous,
sokolov2008statistics}, which is why we are primarily interested in the scenario where $K$ has a power-law decay rate $\alpha \in (0,1]$.       

\subsection{Paper Overview}
The rest of the paper is organized as follows.  In Section~\ref{sec:assumption}, we introduce assumptions and briefly state the well-posedness result for~\eqref{eqn:GLE}. In Section~\ref{sec:pathspace}, we discuss the solutions' structures in accordance to different assumptions on the memory kernel and the noise. In particular, we will see that the dynamics~\eqref{eqn:GLE} induces a skew-flow on the \emph{skew} path space.  Section~\ref{sec:invariant-measure} discusses the associated stationary solution(s) for this dynamics.  Furthermore, we prove our main result on the uniqueness of the associated stationary measures in this section. The argument proving uniqueness, in particular, makes use of some auxiliary results collected and proved in Section~\ref{sec:auxil-result}. In Section~\ref{sec:existence}, we establish the existence of a stationary measure when the kernel can be written as an infinite sum of exponentials. In Appendix~\ref{sec:well-posed}, we establish the well-posedness result in detail. In Appendix~\ref{sec:F}, we prove a technical result which allows us to bound the expected value of the maximum of $F(t)$ over finite intervals of time.  This result is employed in the proof of well-posedness.

\section{Assumptions and well-posedness} \label{sec:assumption}

\subsection{Well-posedness} \label{sec:wellposed}
We begin by clarifying what we mean by a solution of
\eqref{eqn:GLE}. Throughout, we consider a probability space
$(\Omega,\mathcal{F}, \P,\{\mathcal{F}_t\})$ where the set $\Omega$ is endowed with a probability measure $\P$ and  a 
filtration of sigma-algebras $\{\mathcal{F}_t\colon t \in \rbb\}$.

\begin{definition}[Solution on $(-\infty,\infty)$] \label{def:infiniteSolution}A \emph{(weak) solution} to \eqref{eqn:GLE}  on the time
  interval $(-\infty,\infty)$ is a
  probability space $(\Omega,\mathcal{F}, \P,\{\mathcal{F}_t\})$ on which a triple of stochastic processes $(\xi,F,W)$ is defined so that the following conditions are satisfied:
  \begin{enumerate}
  \item   $\xi(t)=(x(t),v(t))$, $F(t)$ and 
    $W(t)$ are all stochastic processes adapted to the filtration $\{\mathcal{F}_t\}$.
  \item $F(t,\omega)$ is a stationary Gaussian process with mean zero and
    covariance $K$ in the sense of \eqref{eq:cov} and  $W(t,\omega)$ is a standard, two-sided Brownian Motion both
    with respect to $\{\mathcal{F}_t\}$ such that $F$ and $W$ are independent.
  \item   With probability one, the triple $(\xi,F,W)$ solves \eqref{eqn:GLE}; that is, with probability one, for
    all $t_0, t_1 \in \rbb$ with $t_0<t_1$ we have
    
    \begin{equation}
      \label{eq:SolveIntEquation}
    \begin{aligned}
      x(t_1)-x(t_0)&=\int_{t_0}^{t_1}v(t)\,dt,\\
v(t_1)- v(t_0)&=-\int_{t_0}^{t_1}\Big[v(t)\, + \U'(x(t)) \,+ \int_{-\infty}^t K(t-r)v(r)\, dr \Big]\, dt \\&\qquad+  \sqrt{2} \big(W(t_1)-W(t_0)\big)+ \int_{t_0}^{t_1} F(t)\, dt.
\end{aligned}
    \end{equation}
  \end{enumerate}
\end{definition}

 \begin{definition}[Solution with an initial past] \label{def:initialSolution} A (weak) solution to \eqref{eqn:GLE}  on the time
  interval $(T_0,T_1)$ with $T_0 \in \rbb$ and $T_1 \in \rbb \cup
  \{\infty\}$ with \textit{initial past} $\xi_0=(x_0,v_0) \in
  C((-\infty,T_0];\rbb^2)$ satisfies the same conditions as in
  \cref{def:infiniteSolution} but the stochastic processes need only
  be defined on the time interval $(T_0,T_1)$ with the exception of
  $\xi=(x,v)$ which is defined on $(-\infty,T_1)$ with
  $\xi(t)=\xi_0(t)$ for $t \in (-\infty,T_0]$. Additionally,
  \eqref{eq:SolveIntEquation} need only hold for $t_0,t_1 \in (T_0,T_1)$.  
\end{definition}

\begin{remark}
In this paper, we will prove strong existence of solutions on $[T_0, \infty)$ given an initial past $\xi_0=(x_0, v_0)$ belonging to an appropriate subclass of $C((-\infty, T_0]; \rbb^2)$.  Moreover, we will also establish weak uniqueness, which together is stronger than weak existence and weak uniqueness.   
\end{remark}

Throughout, we will employ the following assumption on the potential $\U$ in~\eqref{eqn:GLE}.   
\begin{assumption}  \label{cond:Phi}  The potential $\U\colon \mathbb{R}\rightarrow \mathbb{R}$ is such that $\U\in C^3(\rbb)$, $\int_\rbb |\U'(x)|e^{-\U(x)}dx<\infty$ and the global estimate holds
\begin{align*}
b(\U(x)+1)\ge |x|^{1+\delta} \,\, \text{ for all }\,\, x\in \mathbb{R},
\end{align*}
for some constants $b>0$ and $\delta\in(0,1)$.  \end{assumption}
\begin{remark}
The first two conditions on $U$ are not directly used in this paper. They were previously used in \cite[Theorem 7]{glatt2018generalized} to construct an explicit invariant measure for the Markov system~\eqref{eqn:GLE-Markov:xvz} below. We then will use this result to construct a stationary measure for the dynamics~\eqref{eqn:GLE} in Section~\ref{sec:existence}.
\end{remark}
We also use the following condition on the memory kernel.
\begin{assumption} \label{cond:K:supK(t+s)/K(s)}
$K\in C^1([0, \infty); [0, \infty))$ and there exists $\widetilde{K}\in C([0, \infty))$ for which
\begin{align*}
\sup_{s\ge 0}\frac{K(t+s)}{K(s)}= \widetilde{K}(t) \,\,\, \text{ for all } \,\,\, t \geq 0. 
\end{align*}
\end{assumption} 
In order to state our main existence and uniqueness result, for $t\in \rbb$ let 
\begin{equation} \label{eqn:Cpast:int.K'(r).x(r)<infty}
\C(-\infty,t] := \bigg\{ (x,v)\in C((-\infty,t];\rbb^2) \, : \, \int_{-\infty}^{t}\close K(t-r)|v(r)|dr<\infty\bigg\}.
\end{equation}

\begin{proposition} \label{prop:well-posedness}
  Suppose that Assumption~\ref{cond:Phi} and Assumption~\ref{cond:K:supK(t+s)/K(s)} are satisfied.  Then there exists a subset $\K \subset C((-\infty,\infty);\rbb)$ so that
  $\P( F \in \K)=1$ and for every $t_0 \in \rbb$, $F \in \K$ and
every initial condition $\xi_0=(x_0,v_0)\in\C(-\infty,t_0]$, there exists a unique solution $\xi=(x,v)$ with initial
 past $\xi_0$ on the time interval $[t_0, \infty)$ such that $\xi\in \C(-\infty,t]$ for all $t\geq t_0$. Furthermore, we have the energy estimate
 \begin{multline}\label{ineq:energy-estimate}
 \E\sup_{t_0\leq r\leq t}H(x(r),v(r)) \\
 \leq \bigg[H(x_0(t_0),v_0(t_0))+\bigg( \displaystyle\int_{-\infty}^{t_0} \close K(t_0-r)|v_0(r)|dr\bigg)^2+\E\sup_{t_0\le r\le t}F(r)^2+ 1\bigg] e^{c(t_0,t)}, 
 \end{multline}
where we recall that $H(x,v)= \frac{1}{2}v^2+ \U(x)$.  
\end{proposition}

The proof of Propoosition~\ref{prop:well-posedness} is given later in Appendix~\ref{sec:well-posed}.  

\begin{remark}
For a general centered stationary Gaussian process $F(t)$, it is not immediately obvious that for all $t_0<t$
\begin{align} \label{ineq:E.sup.F_t^2<infty}
  \E\sup_{t_0\le r\le t}F(r)^2<\infty.
\end{align} 
In Appendix~\ref{appendix:F}, we will make use of the condition that $K\in C^1$, cf. Assumption~\ref{cond:K:supK(t+s)/K(s)}, to show that this is indeed the case for the process $F(t)$. 
\end{remark}

\subsection{Structural assumptions on the noise}

At times, we will further assume that memory kernel $K$ has the following
specific form previously employed in~\cite{glatt2018generalized}.
\begin{assumption}\label{a:sumExpK}
There exists continuously differentiable functions $J_\ell \colon [0, \infty) \rightarrow [0,
\infty), \ell \geq 1,$ so that the stationary Gaussian forcing
$F(t)$ can be represented as
\begin{align}\label{eq:FintB}
  F(t) = \sum_{\ell = 1}^\infty \int_{-\infty}^t\close J_\ell(t-s) dB^{(\ell)}(s),
\end{align}
where $\{B^{(\ell)} : \ell \geq 1\}$ is a collection of mutually
independent standard two-sided Brownian motions. Furthermore,
\begin{align*}
 t \mapsto  \sum_{\ell = 1}^\infty \int_{0}^\infty\close J_\ell(t+r)J_\ell(r) dr ,
\end{align*}
is continuously differentiable.  
\end{assumption}
\begin{remark}
\label{rem:Kform} 
   \cref{a:sumExpK} together with the fluctuation-dissipation relation~\eqref{eq:cov} immediately imply that
   the memory kernel $K(t)$ is continuously differentiable and of the form
   \begin{align*}
     K(r) = \sum_{\ell = 1}^\infty K_\ell(t) \quad\text{where}\quad K_\ell(t)=\int_{0}^\infty\close J_\ell(t+r)J_\ell(r) dr.   
   \end{align*}
\end{remark}

We will also need some structure on the decay of the kernel at infinity.  
\begin{assumption}\label{cond:K:K(t)<t^-alpha}
There exist constants $t_*>0$, $C>0$ and $\alpha>1/2$ such that 
\begin{align*}
K(t)\le  Ct^{-\alpha} \,\,\, \text{ for all } \,\,\, t\geq t_*.  
\end{align*}
\end{assumption}
\begin{remark}
\label{rem:specialform}
When $F$ is of the form~\eqref{eq:FintB}, an example of particular interest is when $J_\ell$, $\ell \geq 1$, is given by
  \begin{align*}
    J_\ell(t) = \sqrt{2 c_\ell \lambda_\ell} e^{-\lambda_\ell t},
  \end{align*}
  where  
\begin{equation} \label{eqn:c_k}
c_\ell=\frac{1}{\ell^{1+\alpha\beta}} \qquad\text{ and }\qquad\lambda_\ell=\frac{1}{\ell^\beta},
\end{equation}
for some constants $\alpha>0, \beta >1$.  In this case, 
\begin{equation} \label{eqn:K}
K(t)=\sum_{\ell= 1}^\infty c_\ell e^{-\lambda_\ell t},
\end{equation}
and one can show that \cite[Example 3.2]{abate1999infinite}
\begin{align*}
K(t)\sim t^{-\alpha},\quad t\to\infty.
\end{align*}  
Hence, $K$ is a power-law memory kernel which clearly satisfies Assumptions~\ref{cond:K:supK(t+s)/K(s)} and \ref{cond:K:K(t)<t^-alpha}.
\end{remark}

\begin{remark}
Note that if we first suppose that $K$ is of the form~\eqref{eqn:K}, Doob's Theorem \cite{doob1942brownian} and the fluctuation-dissipation relation~\eqref{eq:cov}
together imply that $F$ must be of the form
\begin{equation} \label{eqn:F}
F(t)=\sum_{\ell= 1}^\infty \sqrt{2\lambda_\ell c_\ell}\int_{-\infty}^t \close e^{-\lambda_\ell(t-r)}dB^{(\ell)}(r),
\end{equation}
where in the above, $\{B^{(\ell)}\}_{\ell\geq 1}$ are two-sided, independent
standard Brownian motions.
\end{remark}

When \cref{a:sumExpK} holds, we arrive at the following form for the GLE 
\begin{equation} \label{eqn:GLE:original}
\begin{aligned}
d\, x(t)&=v(t)\,dt,\\
d\, v(t)&=-v(t)\,dt-\U'(x(t))\,dt- \sum_{\ell\geq
  1}\int_{-\infty}^t \close K_\ell(t-r) v(r)dr\, dt \\
&\qquad\qquad+\sum_{\ell\geq 1}\int_{-\infty}^t \close J_\ell(t-r) dB^{(\ell)}(r)\, dt +\sqrt{2}\, dW(t),
\end{aligned}
\end{equation}
where $W$ is a standard, two-sided, real-valued Brownian motion independent
of the collection $\{B^{(\ell)}\}_{\ell\geq 1}$ and $K_\ell$ is as in Remark~\ref{rem:Kform}.

\section{Structures on Pathspace}
\label{sec:pathspace}
Since we often work on the phase space $C(\rbb; \mathbb{R}^2)$ and its subspaces, we use the topology on $C(\rbb; \mathbb{R}^2)$ defined in the follow sense: A sequence $\{g_n\}\subset C(\rbb; \mathbb{R}^2)$ is said to \emph{converge} to $g\in C(\rbb;\rbb^2)$ if the convergence holds in the sup norm on any bounded time interval. That is, for all fixed $T>0$,
\begin{align*}
\sup_{t\in[-T,T]} |g_n(t)-g(t)|\to 0,\quad\text{as } n\to \infty.
\end{align*}
The closed sets in $C(\rbb; \rbb^2)$ are then defined with respect
to the above mode of convergence, hence inducing the corresponding
topology of open sets as well as the Borel sigma algebra of subsets of
$C(\rbb; \rbb^2)$. 

In the introduction, we already discussed how~\eqref{eqn:GLE} along
with its two forcings, $W$ and $F$, can be
viewed together as Markov process on the extended path space $C(\rbb;\rbb^4)$
under the shift map. However, this encodes little useful structure of the system. This is in direct contrast to the more
traditional Markovian embeddings  which hold when \eqref{eqn:Kexpr}
(with $n$ possibly infinite) is enforced as also discussed in the introduction.  In this section, we therefore discuss some intermediate, but fruitful structures used in later sections in this paper.  To aid in the discussion, we begin with a number of preliminary
discussions in simplified settings.

\subsection{The structure of solutions in simplified settings}\label{sec:simplfiedSetting}

As we have already noted, when both the general Gaussian forcing $F$
and  memory kernel $K$ are taken to be zero, \eqref{eqn:GLE} is a
standard stochastic differential equation (SDE) which generates a Markov process on $\rbb^2$. The
appearance of each of these introduces particular complications and
structures. We will first consider them individually before
exploring their combined effects.
\subsubsection*{Time inhomogeneous SDE and its skew-flow of kernels}
If only the memory kernel $K$ is taken to be zero and $F$ is a stationary Gaussian process, then the resulting equation~\eqref{eqn:GLE} is a
standard, time-homogeneous SDE.  The resulting equation, in particular, generates (provided solutions make sense) a family of solution maps $\varphi_{s,t}^{F,W}$ of~\eqref{eqn:GLE} for $(F,W) \in C(\rbb; \rbb^2)$ and $s \leq t$. The addition of $F$ does not destroy the classical
skew structure of the SDE; namely, 
\begin{align*}
\varphi_{s+r,t+r}^{F,W}
=\varphi_{s,t}^{\theta_r(F,W)} \text{ for any }r \in \rbb
\end{align*}
 where, for any
function of time $f$,  we offer the slight abuse of notation and set $(\theta_rf)(t)= f(t+r)$.

By averaging over $W$, we define a flow of Markov
kernels $R_{s,t}^F$ by
\begin{align*}
 R_{s,t}^F(\xi, A) =
\P(\varphi_{s,t}^{F,W}(\xi)\in A | \xi, F) \text{ for }A \subseteq \rbb^2
\end{align*}
and initial conditions $\xi=(x,v) \in \rbb^2$. For $s< r<t$, we have the
usual time inhomogeneous Markov property $R_{r,t}^FR_{s,r}^F=
R_{s,t}^F$. But, we also have the following skew property inherited
from the underline SDE, $$R_{s+r,t+r}^F=R_{s,t}^{\theta_r F}$$ for
$s\leq t$ and $r \in \rbb$. (See Section~\ref{sec:Skew} for more details.)

Without more information on $F$, Markovian representations of the
dynamics must include the entire future of the process $F$. This means
that the only independence of the future from the past must come from
the standard Brownian motion $W$ and not the process $F$. However, if
the process $F$ satisfies Assumption~\ref{a:sumExpK}, then there is
memory loss in $F$ and one can define Markov process with state
variables $(\xi, B^{(1)},B^{(2)},\dots)$ on the statespace
$\rbb^2 \times C( (-\infty,0] ,
\rbb)^\nbb$.  The resulting Markov Kernel $P_t$ is then defined as follows: First $\{B^{(n)}\}_{n=1}^\infty$ are extended to the time intervals
$(-\infty,t]$ by drawing independent increments of the Weiner process
of each. Next, a realization of the Weiner process $W$ is drawn on the
time $[0,t]$ starting from zero as only the increments of $W$ are used. Then the initial conditions are  evolved to time $t$ using
$\varphi_{0,t}^{\theta_r(F,W)}(\xi)$ with $F$ reconstructed from 
$( B^{(1)},B^{(2)},\dots)$ using the formula from \eqref{eqn:F}. The
resulting state $(\xi(t), B^{(1)},B^{(2)},\dots)$ is a random
variable taking values in $\rbb^2 \times C( (-\infty,t] ;\rbb)^\nbb$
and hence  $(x(t),v(t), \theta_tB^{(1)},\theta_tB^{(2)},\dots)$ takes
values again in $\rbb^2 \times C( (-\infty,0] ;\rbb)^\nbb$. The law of
this random variable is taken as the transition measure defining $P_t$
starting from this initial condition. The advantage of this
representation is that the marginals the process in $\xi=(x,v)$ again have
``Markovian feel''  of the original process.

\subsubsection*{The Gibbsian SDE and Markov process on path space}
Consider now the situation where $F\equiv 0$ and we leave the memory kernel
intact. The resulting dynamics is not a Markovian diffusion in the
classical sense. However, the resulting SDE is still a rather standard
It\^o process 
as its coefficients at time $t$  are still adapted to the past of
$W$. This particular form of  an It\^o process, considered in 
\cite{bakhtin2005stationary,weinan2002gibbsian,Mattingly11,ito1964stationary,
  Mattingly02,
  Mattingly03},
can be written abstractly as
\begin{equation}\label{eq:gibbsSDE}
  d\xi(t)= f( \theta_t \xi_{(-\infty,t]}) dt + g(\theta_t \xi_{(-\infty,t]})dW(t), 
\end{equation}
where $\theta_t$ is again the shift in time on pathspace, $\xi(t) \in \mathbb{R}^d$, $\xi_{(-\infty,t]} \in C((-\infty,t]; \mathbb{R}^d)$ , $\theta_t 
\xi_{(-\infty,t]} \in C((-\infty,0]; \mathbb{R}^d)$ and $f, g\colon 
  C((-\infty,0]; \mathbb{R}^d) \rightarrow \mathbb{R}^d$
  are the coefficients of the process. In the context of \eqref{eqn:GLE}
  with $F\equiv 0$, the dimension $d$ is 2, $f$ represents the drift terms in
  \eqref{eqn:GLE} (including the memory term), and $g$ is the constant
  $2\times2$ matrix with $g_{22}=\sqrt{2}$ and all other entries zero.

In
\cite{bakhtin2005stationary,weinan2002gibbsian,Mattingly11,Mattingly02,
Mattingly03},
this type of equation was termed Gibbsian in that it
defined a family of compatible conditional transition kernels which
depend on the entire past of the process rather than the most recent point in time as
in the Markovian setting.
This process has an infinitesimal Gibbsian generator at time $t$
given by
\begin{align*}
  L_{\xi_{(-\infty,t]}}^th(\xi(t)) = \sum_{i=1}^df_i(\xi_{(-\infty,t]} )\partial_i
  h( \xi(t)) + \frac12 \sum_{i,j=1}^d a_{ij}(\xi_{(-\infty,t]} )
  \partial_i \partial_j h(\xi(t)),
\end{align*}
for a test function $h : \rbb^d
\rightarrow \rbb$ and matrix $a=g g^T$.  This structure implies a certain amount of
independence of the future from the past, or at least a rate
of decorrelation depending on the properties of $K$.  In this case, we can define a
family of random maps $\varphi_{s,t}^W\colon C((-\infty,s];\rbb^d)
\rightarrow  C((-\infty,t];\rbb^d)$ for $s \leq t$ depending on a
random increment path of Brownian motion $W$ of length $t-s$. For
sufficiently nice $h : \rbb^d \rightarrow \rbb$, we have that
\begin{align*}
    L_{\xi_{(-\infty,t]}}^th(\xi(t))= \lim_{r \rightarrow 0^+} \frac1{r}\E \big[ h\big(\varphi^W_{t,t+r}(\xi_{(-\infty,t]})(t+r)\big) - h(\xi(t))\big]\,.
\end{align*}
By setting $P_t(\xi_{(-\infty,0]} , \;\cdot\;)$ to be the law of $\theta_t
\varphi^W_{0,t}(\xi_{(-\infty,0]})$ viewed as a random variable taking values in
$C((-\infty,0];\rbb^d)$, we can define a Markov operator on the space $C((-\infty,0];\rbb^d)$.  This Markovian representation  has more
structure than the lifting to the future and past performed in the
introduction as it encodes that the future in our context only depends on the
past.

\subsection{The Skew-flow and kernel for the full system}\label{sec:Skew}
We now combine the discussions above to provide insight into the
structure of \eqref{eqn:GLE} when both $K$ and $F$ are nonzero.  We
will reuse the symbols $\varphi$ and $R$ from the previous section but
with sightly different domains of definition needed to accomodate our
current setting with neither $F$ nor $K$ identically zero. We allow
this slight abuse of notation to make the analogies between this section
and the preliminary discussion in Section~\ref{sec:simplfiedSetting} above clearer.

As
before, we can associate to the dynamics~\eqref{eqn:GLE} a
skew product flow; however, we now must include the past of $x$ and
$v$ because of the memory in the drift. That is, given a realization of $F$ and $W$ on the
time interval $(-\infty,\infty)$ and recalling the space $\C(-\infty,t]$ as in \eqref{eqn:Cpast:int.K'(r).x(r)<infty}, we define the family maps 
\begin{align}
\varphi_{t_0,t}^{F,W}\colon \C(-\infty,t_0] \rightarrow \C(-\infty,t], \,\, t_0\leq t, 
\end{align}
 as the
extension of an initial past $\xi_0 \in  \C(-\infty,t_0]$ to a function
in $\C(-\infty,t]$ by appending to the front of $\xi_0$ the solution
\eqref{eqn:GLE} on the time interval $[t_0,t]$ with initial
past $\xi_0$ and random forcing $W$ and $F$. When $\xi_0$ is deterministic,
$\varphi_{t_0,t}^{F,W} \xi_0$ is a random path adapted to
\begin{align*}
\mathcal{F}_{t_0,t}=\sigma( F(r), W(r)-W(t_0) : r \in [t_0,t])  \,\,\,\,\, \text{with}\,\,\,\,\, (\varphi_{t_0,t}^{F,W} \xi_0)(r)=\xi_0(r) \,\, \text{ for }\,\,r\leq t_0.
\end{align*}
Observe that if $\theta_t$ again denotes the shift map in the space of
trajectories, defined by  $\theta_t f(s)=f(s+t)$, then $\theta_t \varphi_{0,t}^{F,W} \xi_0\colon \C(-\infty,0]
\rightarrow \C(-\infty,0]$. More specifically, we see that
\begin{center}
\begin{tabular}{ccccc}
$\mathcal{C}(-\infty,0]$ & $\rightarrow$ & $\mathcal{C}(-\infty,t]$ & $\rightarrow$ & $\mathcal{C}(-\infty,0]$\\
$\xi_0$ & $\mapsto$  & $\f^{F,W}_{0,t}\xi_0$ & $ \mapsto$ & $ \theta_t \f^{F,W}_{0,t}\xi_0$. 
\end{tabular}
\end{center}
\noindent So, the skew-flow $S_t$ defined by
\begin{align}\label{eqq:skewFlow}
  S_t\colon (\xi_0,F,W) \mapsto ( \theta_t \varphi_{0,t}^{F,W} \xi_0,
  \theta_t F, \theta_t W-W(t)),
\end{align}
is a random semi-flow on the space $\C(-\infty,0] \times
C(\rbb; \rbb^2)$. In particular $S_{s+t}=S_s S_t$.

Next we define the skew transition kernel $R_t^F$ on  $\C(-\infty,0]$
by taking the law of $\theta_t \varphi_{0,t}^{F,W} \xi_0$ conditioned
on $\xi_0$ and $F$; namely,
\begin{align} \label{form:R^F_t}
  R_t^F(\xi_0, A) := \P( \theta_t \varphi_{0,t}^{F,W}\xi_0 \in A \, |\, 
  \xi_0, F) ,
\end{align}
for 
\begin{equation} \label{form:S_skew}
(\xi_0,F) \in \S_{\textrm{skew}}:= \C(-\infty,0] \times C(
(-\infty,\infty);\rbb),
\end{equation}
and $A \subset \C(-\infty,0]$ Borel. Observe we
have the following skew structure stemming from \eqref{eqq:skewFlow}
\begin{align*}
 R_t^F   R_s^{\theta_t F}= R_{t+s}^F, 
\end{align*}
or more explicitly,
\begin{align*}
  R_{t+s}^F(\xi_0,A) = \int_{\C(-\infty,0]}\close R_t^F(\xi_0,d\zeta)  R_s^{\theta_t F}(\zeta,A) . 
\end{align*}

\subsection{A more Markovian kernel}\label{eq:KernalMarkov}
Looking at \eqref{eqn:GLE:original}, we see that when \cref{a:sumExpK}
is enforced, we can consider a solution to be a triple of stochastic processes $(\xi,W,B)$
where $\xi$ and $W$ are as before but $B=\{B^{(\ell)}\}_{\ell\geq 1}$ is a
countable collection of standard two-sided independent Brownian Motions independent of $W$. We can then define a
map $\psi_t^W: \S(-\infty,0]  \rightarrow \S(-\infty,t]$ where  
\begin{align} \label{form:S(-infty,t]}
\S(-\infty,t] := \C(-\infty,t] \times
C((-\infty,t];\rbb)^{\nbb},
\end{align}
and $\psi_t^W(\xi_0,B_0)$ is equal to  
the pair $(\xi,B)$ obtained by continuing the Brownian motions $B_0\in C((-\infty,0];\rbb)^{\nbb}$ over
the interval $[0,t]$ and then extending $\xi$ over the same interval
by evolving \eqref{eqn:GLE:original} using $F$ as in \eqref{eq:FintB}
with the Brownian Motions in $B$, which is the continuation of $B_0$. We again have a skew-flow defined
for $(\xi_0,B_0) \in \S(-\infty,0]$ and $W \in C((-\infty,\infty); \rbb)$
by
\begin{align*}
  (\xi_0,B_0,W) \mapsto ( \theta_t\psi_t^W(\xi_0,B_0), \theta_t W)\,.
\end{align*}
In contrast to Section~\ref{sec:Skew}, where the skew-flow $S_t$ as in~\eqref{eqq:skewFlow} is fibered over the bivariate process $(W,F)$, whose future increments depend on its entire past, this skew-flow $(\theta_t\psi_t^W(\xi_0,B_0),\theta_tW)$
is fibered over a process, namely $W$, whose future increments are independent of
its past increments. Thus, we can obtain a Markov kernel by
averaging over the randomness in $W$.  We cannot average over the
randomness in $B$ as the increment added to $\xi$ over the time
interval $[0,t]$ depends on the entire history of $B$ back to time
$-\infty$. 

With these considerations, we define the Markov kernel
$P_t$ on $\S(-\infty,0]$ by
\begin{align}\label{eq:MarkovOnC}
  P_t\big((\xi_0,B_0), A\big) = \P( \theta_t\psi_t^W(\xi_0,B_0) \in A \, |\,\xi_0,B_0  ),
\end{align}
for $(\xi_0,B_0) \in \S(-\infty,0]$ and  $A \subset  \S(-\infty,0]$
Borel.

\subsection{Solutions on the time interval $[0,\infty)$}

\cref{prop:well-posedness} gives a finite-time existence and uniqueness 
result for initial pasts in 
$\C(-\infty,0]$. Thus solutions do not blow up in finite time, but it is possible that they
may tend to $\infty$ as $t \rightarrow \infty$. Hence this fact induces a well-defined mapping
\begin{align*}
\varphi_\infty^{F,W}\colon \C(-\infty,0] \rightarrow C([0,\infty);\rbb^2),
\end{align*} 
but it is still possible that  
\begin{align*}
\P(|\varphi_\infty^{F,W}(\xi_0)(s)|
\rightarrow \infty \text{ as } s \rightarrow \infty)>0. 
\end{align*}
In the next section, we will
consider the large-time behavior of the system, in particular the existence and uniqueness of stationary solutions.  

Because the mapping $\varphi_\infty^{F,W}$ makes sense, we can define a family 
of kernels $Q^F_{[0,\infty)}$ on the  \emph{infinite future} by
  \begin{align}\label{eq:QDef}
    Q^F_{[0,\infty)}(\xi_0, A) = \P( \pi_{[0,\infty)}\varphi_\infty^{F,W}(\xi_0) \in A \,|\,
    F, \xi_0),
  \end{align}
for initial pasts  $\xi_0 \in \C(-\infty,0]$ and Borel sets $A \subset
C([0,\infty); \rbb^2)$.  Here, $\pi_{[0,\infty)}$ denotes the projection of
the trajectory onto the time interval $[0,\infty)$. While $R^F_t$ captures the effect of starting
from an initial past at time $-t$ and flowing forward to time $0$, $Q^F$
captures the distribution on the infinite future starting from $\xi_0$
at time 0.

\section{Stationary Solutions and Invariant Measures}
\label{sec:invariant-measure}
Recall that the stochastic  process $(\xi,F)$ on the time interval $(-\infty,\infty)$ is \textit{stationary}
if for any finite collection of  times $t_1,\cdots,t_n \in (-\infty, \infty)$ the
distribution of the random vector
$$\big((\xi(t_1+s),F(t_1+s)),\ldots,  (\xi(t_n+s),F(t_n+s))\big),\quad t_1<t_2<\cdots<t_n,$$ is
independent of $s \in \rbb$. Letting $\theta_t$ denote the shift mapping on the space of trajectories, cf.~\eqref{form:theta_t}, stationarity is equivalent in our setting to
the distribution of the path $\theta_t(\xi,F)$ being independent of 
$t$. 

In Sections~\ref{sec:SkewInvMeasure} and \ref{sec:MarkovInvMeasure} below, we discuss the relation of stationary solutions and invariant measures for the skew-kernel $R^F_t$ and $P_t$ defined in~\eqref{form:R^F_t} and~\eqref{eq:MarkovOnC}, respectively.

\subsection{For the skew kernel $R_t^F$}\label{sec:SkewInvMeasure} 

Recalling the skew transition kernel $R^F_t$ on $\C(-\infty,0]$ as in~\eqref{form:R^F_t}, a family of measures $\mu^F$ on Borel subsets of $\C(-\infty,0]$, indexed by a realization $F \in C((-\infty,\infty);\rbb)$, is called \emph{skew-invariant} for $R^F_t$ if the following holds
\begin{align*}
  \mu^F R_t^F = \mu^{\theta_t F}, 
\end{align*}
where we define the measure  $\mu^F R_t^F$ for $A\subset \C(-\infty, 0]$ Borel by
\begin{align}
  \mu^F R_t^F(A)& = \int_{\C(-\infty,0]}\close\close R_t^F(\xi_0,A)  \mu^F(d\xi_0)\notag\\
  &=\int_{\C(-\infty,0]}\close\close\P( \theta_t \varphi_{0,t}^{F,W}\xi_0 \in A \, |\, 
  \xi_0, F) \mu^F(d\xi_0)\,. \label{eqn:skew-inv}
\end{align}

We note that a stationary solution $(\xi,F)$ on the time interval
$(-\infty,\infty)$ always generates a skew-invariant measure
$\mu^F$. To see this, let $\Law(\xi,F)$ denote the law of the stationary solution $(\xi,F)$. Then, Law($\xi$), being the disintegration of $\Law(\xi,F)$ relative to $\Law(F)$ restricted
to $\C(-\infty,0]$, is the desired skew-invariant $\mu^F$. Indeed, from~\eqref{eqn:skew-inv}, observe that $\mu^FR^F_t$ is the law of $\theta_t \xi$, which agrees with Law($\xi$) by stationarity.

On the other hand, given a skew-invariant measure
$\mu^F$ on $\C(-\infty,0]$, let $\tilde \mu^F$ be the extension of
$\mu^F$ to the time interval $(-\infty,\infty)$ using the dynamics
$\varphi_{0,t}^{F,W}$. That is, for any Borel set $A\in \C(-\infty,t]$,
\begin{align*}
\tilde{\mu}^F(A)=\int_{\C(-\infty,0])}\close\close\P(  \varphi_{0,t}^{F,W}\xi_0 \in A \, |\, 
  \xi_0, F) \mu^F(d\xi_0)\,.
\end{align*}
Then $\tilde \mu^F (d\xi) \Law(F)(df) $ is the law of the
desired stationary process $(\xi,F)$. To see the stationarity of $\xi$, for $A_1,\dots,A_N\subset\rbb^2$,
\begin{align}
&\tilde{\mu}^F\big(\big\{ \xi(t_1+s)\in A_1,\dots, \xi(t_n+s)\in A_n   \big\}\big)\notag\\
&=\int_{\C(-\infty,0])}\close\close\P(  \varphi_{0,t_n+s}^{F,W}\xi_0(t_1+s)\in A_1,\dots, \varphi_{0,t_n+s}^{F,W}\xi_0(t_n+s)\in A_n \, |\, 
  \xi_0, F) \mu^F(d\xi_0)\notag\\
  &= \int_{\C(-\infty,0])}\close\close\P(  \theta_{t_n+s}\varphi_{0,t_n+s}^{F,W}\xi_0(t_1-t_n)\in A_1,\dots, \theta_{t_n+s}\varphi_{0,t_n+s}^{F,W}\xi_0(0)\in A_n \, |\, 
  \xi_0, F) \mu^F(d\xi_0)\notag\\
  &=\mu^F\big(\big\{ \xi'(t_1-t_n)\in A_1,\dots, \xi'(0)\in A_n   \big\}\big)\notag\\
  &= \int_{\C(-\infty,0])}\close\close\P(  \theta_{t_n}\varphi_{0,t_n}^{F,W}\xi_0(t_1-t_n)\in A_1,\dots, \theta_{t_n}\varphi_{0,t_n}^{F,W}\xi_0(0)\in A_n \, |\, 
  \xi_0, F) \mu^F(d\xi_0)\notag\\
  &=\int_{\C(-\infty,0])}\close\close\P(  \varphi_{0,t_n}^{F,W}\xi_0(t_1)\in A_1,\dots, \varphi_{0,t_n}^{F,W}\xi_0(t_n+s)\in A_n \, |\, 
  \xi_0, F) \mu^F(d\xi_0)\notag\\
  &=\tilde{\mu}^F\big(\big\{ \xi(t_1)\in A_1,\dots, \xi(t_n)\in A_n   \big\}\big). \label{eqn:skew-inv:mutilde}
\end{align}
In the third and fourth implications above, we invoked the stationarity of $\mu^F$.


\subsection{For the Markov kernel $P_t$}\label{sec:MarkovInvMeasure}

When \cref{a:sumExpK} holds, recall that $(\xi,B)$ evolves as a Markov
process on the state space
\begin{equation*} 
\S(-\infty,0] :=\C(-\infty,0] \times C((-\infty,0];\rbb)^{\nbb},
\end{equation*}
under the Markov kernel $P_t$ defined in \eqref{eq:MarkovOnC}. In this
setting, there
is a one-to-one correspondence between stationary solutions on the
time interval $(-\infty,\infty)$ and invariant probability measures $\mu$ for $P_t$.

Given an invariant probability
measure $\mu$ for $P_t$ on  $\S(-\infty,0]$, we can create a
stationary measure $\tilde \mu$ on the interval $(-\infty,\infty)$ by flowing the
dynamics forward from $\mu$ by the map $\psi_t^W$ defined in
Section~\ref{eq:KernalMarkov} from a random initial past distributed
according to $\mu$ and then taking the measure obtained on
$(-\infty,\infty)$ by averaging over the realization of $W$. That is, for any Borel set $A\in \S(-\infty,t]$,
\begin{align*}
\tilde{\mu}(A)=\int_{\S(-\infty,0]}\close\close\P(  \psi_{t}^{W}(\xi_0,B_0) \in A \, |\, 
  \xi_0, B_0) \mu(d\xi_0,d B_0)\,.
\end{align*}
The argument for $\tilde{\mu}$ being stationary is analogous to~\eqref{eqn:skew-inv:mutilde}.

Conversely, given a stationary solution $\tilde{\mu}$ on $\S
(-\infty,\infty)$ then we can simply restrict the distribution to a
measure $\mu$ on $\S(-\infty,0]$. For any Borel set $A\in \S(-\infty,0]$, observe that $\theta_{-t}A\in \S(-\infty,t]$. Letting $\pi_T\xi$ denote the projection of $\xi$ onto the interval $T\subset(-\infty,\infty)$, we have
\begin{align*}
\mu P_t(A)&=\int_{\S(-\infty,0]}\close\close\P(  \theta_t\psi_{t}^{W}(\xi_0,B_0) \in A \, |\, 
  \xi_0, B_0) \mu(d\xi_0,d B_0)\\
  &= \int_{\S(-\infty,0]}\close\close\P(  \psi_{t}^{W}(\xi_0,B_0) \in \theta_{-t}A \, |\, 
  \xi_0, B_0) \mu(d\xi_0,d B_0)\\
  &=\tilde{\mu}(\{\xi:\pi_{(-\infty,t]}\xi\in \theta_{-t} A\})\\
  &= \tilde{\mu}( \{\xi:\pi_{(-\infty,0]}\theta_t\xi\in A   \})\\
  &= \tilde{\mu}( \{\xi:\pi_{(-\infty,0]}\xi\in A   \})=\mu(A)\,.
\end{align*}
In the second to last implication above, we employed the stationarity of $\tilde{\mu}$. We therefore see that the resulting measure $\mu$ is invariant
for the Markov Kernel $P_t$. 

\subsection{Existence and uniqueness of stationary measures}
Recalling the space $\C(-\infty, t]$ defined in~\eqref{eqn:Cpast:int.K'(r).x(r)<infty}, for $\varrho >0$ we introduce the following subset of \emph{moderate growth}:
\begin{align} \label{eqn:CpastR^2}
\C_\varrho(-\infty,t]=\bigg\{( x,v) \in \C(-\infty,t] : \sup_{r\leq t}\frac{|x(r)|}{1+|r|^\varrho}<\infty\bigg\}  ,
\end{align}
and define
\begin{align*}
  \C_\varrho(-\infty,\infty) = \bigcup_{n \in \Z, n \geq 0}\C_\varrho(-\infty,n].
\end{align*}

Our main result concerning the existence of an invariant measure for the Markov kernel $P_t$ is the following theorem whose proof is deferred to Section~\ref{sec:existence}.
\begin{theorem} \label{thm:existence}
Suppose that $\U$ satisfies Assumption~\ref{cond:Phi} and that Assumption~\ref{a:sumExpK} is satisfied with the choice of $J_\ell$ as in Remark~\ref{rem:specialform}.  Then there exists an invariant measure $\mu_*$ for $P_t$ defined in \eqref{eq:MarkovOnC}. Moreover, for every $\varrho>0$,
\begin{equation} \label{eqn:mu(C_rho)=1)}
\mu_*(\C_\varrho(-\infty,0])=1.
\end{equation}

\end{theorem}

\begin{remark}
The proof of Theorem~\ref{thm:existence} relies on constructing an explicit invariant measure for an infinite-dimensional auxiliary Markovian system. A good Lyapunov-type estimate for the equation~\eqref{eqn:GLE} which would ensure the abstract existence of such a measure in more generality is currently unavailable. It is thus left as an open problem to determine whether~\eqref{eqn:mu(C_rho)=1)} always holds for any invariant measure $\mu$.
\end{remark}

The following is our uniqueness result which pairs with the existence
result given in \cref{thm:existence}. However, it is worth noting that
the uniqueness result applies in many settings where we do not know that there
exists a stationary measure. 

\begin{theorem} \label{thm:unique}
Suppose that $\U$ satisfies Assumption~\ref{cond:Phi} and that the memory kernel $K$ satisfies Assumptions~\ref{cond:K:supK(t+s)/K(s)} and~\ref{cond:K:K(t)<t^-alpha}. For every $\varrho<\alpha-1/2$, the skew dynamics $S_t$ admits at most one stationary solution $(\xi,F)$ such that $\Law(\xi\,|\,F)$ is supported in $\C_\varrho(-\infty,\infty)$.
\end{theorem}
The following corollary is an immediate result of the
\cref{thm:unique} when we are in the Markovian setting discussed in
\cref{eq:KernalMarkov} and \cref{sec:MarkovInvMeasure}.
\begin{corollary}
When \cref{a:sumExpK} holds in addition to the assumptions of
\cref{thm:unique}, there exists at most one invariant measure
supported on  $$\S_\varrho(-\infty,0] :=\{(\xi,B) \in \S(-\infty,0]: \xi
\in \C_\varrho(-\infty,0] \},$$ for the Markov semigroup on that
space discussed in \cref{eq:KernalMarkov}. 
\end{corollary} 

The proof of \cref{thm:unique} makes use of a coupling argument employed
in~\cite{bakhtin2005stationary,Mattingly15,Mattingly11,
Mattingly02,Mattingly03,
weinan2002gibbsian,weinan2001gibbsian}
to show that starting from two distinct initial history paths, the time averages of their solutions in the future must converge to the same place, hence yielding uniqueness of a given stationary measure.  Two of the main ingredients in the coupling argument are the following two results to be proved in the next section.

\begin{proposition} \label{lem:Lebesgue-equivalent}  Under the hypotheses of Theorem~\ref{thm:unique}, for any
  stationary solution $(\xi,F)$ of~\eqref{eqn:GLE}, the marginal of $\Law(\xi \,|\, F)$ at any fixed
  time $t$ is equivalent to Lebesgue measure on $\rbb^2$. \end{proposition}

\begin{proposition} \label{lem:Q_infinity:equivalent} Under the hypotheses of Theorem~\ref{thm:unique}, let $\xi_0$ and $\widetilde{\xi}_0$ be two
  initial pasts in $\C_\varrho(-\infty,0]$ such that
  $\xi_0(0)=\widetilde{\xi}_0(0)$. Then for almost every realization
  of $F$, the measures $Q^F_{[0,\infty)}(\xi_0, \,\cdot\, )$ and $Q^F_{[0,\infty)}(\widetilde{\xi}_0,\,\cdot\,)$ are equivalent.
\end{proposition}

Given these two results, we can now conclude Theorem~\ref{thm:unique}.

\begin{proof}[Proof of \cref{thm:unique}] We first fix some notation.  Given a set $A\subset C(\rbb;\rbb^2)$, a measure $\nu$ on Borel subsets of $C(\rbb;\rbb^2)$ and a time interval $\mathrm{T}\subset \rbb$, we denote by $\pi_\mathrm{T}A$ and $\pi_\mathrm{T}\nu$ to be respectively the projection of $A$ and $\nu$ on $\mathrm{T}$. In other words, letting $\pi_\mathrm{T}\xi$ be the projection of a trajectory $\xi$ onto the time interval $\mathrm{T}$, we set
$$
\pi_\mathrm{T}A=\{\pi_\mathrm{T} \xi: \xi\in A\},
$$
and for any Borel set $B\subset C\big(\mathrm{T};\rbb^2\big)$, 
$$\pi_\mathrm{T}\nu(B):=\nu(\{\xi\in C(\rbb;\rbb^2):\pi_\mathrm{T}\xi\in B  \}).$$

 Let $(\xi_1,F_1)$ and
  $(\xi_2,F_2)$ be two stationary solutions of~\eqref{eqn:GLE}. Without lost of generality, we
  may assume that $\Law(\xi_1,F_1)$ and $\Law(\xi_2,F_2)$ are ergodic by ergodic decomposition. As discussed in Section~\ref{sec:SkewInvMeasure},  we
  can disintegrate $\Law(\xi_i , F_i)$ into $\Law(\xi_i \, |\,
  F)$ relative to $\Law(F)$ since $\Law(F)=\Law(F_1) = \Law(F_2)$. Letting $\nu_i=\Law(\xi_i\,|\,F)$, $i=1,2$, we aim to prove $\nu_1=\nu_2$ assuming $\nu_1$ and $\nu_2$ are supported in $\C_\varrho(-\infty,\infty)$.

Fixing an arbitrary bounded function $\phi :C([0,\infty); \rbb^2)
   \rightarrow \rbb$ which only depends on some compact set of time, Birkhoff's Ergodic Theorem implies that there exists a 
  set $$A_{i} \subset  \C_\varrho(-\infty,\infty),$$ such that $\nu_{i}(A_{i})=1$ and for
  every $\xi \in A_{i}$,
  \begin{align}\label{eq:avgPhi}
    \lim_{T \rightarrow \infty} \frac{1}{T}\int_0^T \phi( \pi_{[0,\infty)}\theta_t \xi) dt = \int
    \phi(\pi_{[0,\infty)}\xi) \nu_{i}(d\xi)=: \phi_{i}.
  \end{align}
It suffices to prove that $\phi_{1}=\phi_{2}.$ To this end, for each $\zeta\in C\big((-\infty,0];\rbb^2\big)$, we set
$$B_i(\zeta) =\{\pi_{[0,\infty) }z : z \in A_i
,\pi_{(-\infty,0]} z = \zeta\}.$$ 
Since $\nu_i$ is supported in $A_i$, it is clear that
\begin{align*}
\pi_{[0,\infty)}\nu_i(\pi_{[0,\infty)}A_i) = \nu(\{z\in A_i:\pi_{[0,\infty)}z\in\pi_{[0,\infty)}A_i\})=\nu_i(A_i)=1.
\end{align*}
On the other hand, recalling $Q^F_{[0,\infty)}$ is the future law as in~\eqref{eq:QDef}, observe that 
\begin{align*}
 1= \pi_{[0,\infty)}\nu_i(\pi_{[0,\infty)}A_i) &= \int_{
  \pi_{(-\infty,0]}A_i}\close\close  Q_{[0,\infty)}^F\big(\,\zeta,
 \pi_{[0,\infty)}A_i\,\big) \pi_{(-\infty,0]}\nu_i(d\zeta),\\
  &= \int_{
  \pi_{(-\infty,0]}A_i}\close\close  Q_{[0,\infty)}^F\big(\,\zeta,
 \{\pi_{[0,\infty)}z:z\in A,\pi_{(-\infty,0]}z=\zeta\}\,\big) \pi_{(-\infty,0]}\nu_i(d\zeta),\\
 &= \int_{
  \pi_{(-\infty,0]}A_i}\close\close  Q_{[0,\infty)}^F\big(\,\zeta,
 B_i(\zeta)\,\big) \pi_{(-\infty,0]}\nu_i(d\zeta),
\end{align*}
We
then conclude that for almost every $\zeta \in
\pi_{(-\infty,0]}A_i$ with respect to $\pi_{(-\infty,0]}\nu_i$, we see that
\begin{equation} \label{eqn:Q(zeta,B(zeta))=1}
Q_{[0,\infty)}^F\big(\,\zeta,
 B_{i}(\zeta)\,\big) =1.
\end{equation}
In view of \cref{lem:Lebesgue-equivalent}, we know that $\pi_0\nu_1$ and $\pi_0\nu_2$ are both equivalent to Lebesgue measure in $\rbb^2$. So that $\pi_0A_1\cap\pi_0A_2\neq \emptyset$. Together with~\eqref{eqn:Q(zeta,B(zeta))=1}, it follows that there exist $\zeta_1$ and $\zeta_2$ such that $\zeta_1(0)=\zeta_2(0)$ and $Q_{[0,\infty)}^F\big(\,\zeta_i,
 B_i(\zeta_i)\,\big) =1$ for $i=1,2$.  As
 \cref{lem:Q_infinity:equivalent} implies that
 $Q_{[0,\infty)}^F\big(\,\zeta_1,\ccdot\big)$ is equivalent to $Q_{[0,\infty)}^F\big(\,\zeta_2,\ccdot\big)$, we also
 know that $$Q_{[0,\infty)}^F\big(\,\zeta_1, B_2( \zeta_2)\big)=1=Q_{[0,\infty)}^F\big(\,\zeta_2, B_1( \zeta_1)\big),$$ 
 and hence
 \begin{align*}
   Q_{[0,\infty)}^F\big(\,\zeta_i,
 B_1( \zeta_1) \cap  B_2(\zeta_2)\,\big) =1,\quad  \text{for}\, i=1,2.
 \end{align*}
In particular, this implies that $ B_1( \zeta_1) \cap  B_2(\zeta_2) \neq \emptyset$. By the definition of $B_i( \zeta_i)$, there exist $z_i\in A_i$, $i=1,2$ such that $\pi_{(-\infty,0]}z_i=\zeta_i$ and
$$ \pi_{[0,\infty)}z_1=\pi_{[0,\infty)}z_2 \in B_1( \zeta_1) \cap  B_2(\zeta_2), $$
whence for all $t\ge 0$,
$$\pi_{[0,\infty)}\theta_t z_1=\pi_{[0,\infty)}\theta_t z_2.$$ 
As a consequence, we have from \eqref{eq:avgPhi} that
\begin{align*}
  \bar \phi_1= \lim_{T \rightarrow \infty}\frac1T \int_0^T \phi(\pi_{[0,\infty)}\theta_t z_1) dt =\lim_{T \rightarrow \infty}\frac1T \int_0^T \phi(\pi_{[0,\infty)}\theta_t z_2) dt = \bar \phi_2.
\end{align*}
As $\phi$ was from a class of
functions sufficiently rich to determine the laws of $\xi^{(i)}$, $i=1,2$,
we conclude the laws are the same since we have proven that $\bar \phi^{(1)}= \bar \phi^{(2)}$.
\end{proof}


 \section{Proofs of \cref{lem:Lebesgue-equivalent} and \cref{lem:Q_infinity:equivalent}} \label{sec:auxil-result}
In order to setup the proof of Proposition~\ref{lem:Q_infinity:equivalent}, observe that we may express equation~\eqref{eqn:GLE} in a convenient form using integration-by-parts on the convolution term.  Indeed, by Assumption~\ref{cond:K:K(t)<t^-alpha}, there exist constants $C, t_*>0$ and $\alpha >1/2$ such that $K(t)\leq C t^{-\alpha}$ as $t\geq t_*$.  Since $K$ is continuously differentiable, L'Hospital's rule implies that 
for any $\epsilon >0$, $K'(t) t^{\alpha +1 -\epsilon} \rightarrow 0$ as $t\rightarrow \infty$.  Now, given that $\xi_0=(x_0,v_0)\in\C_\varrho(-\infty,0]$ where $\varrho<\alpha-1/2$, using integration-by-parts we may thus rewrite \eqref{eqn:GLE} as
\begin{equation}\label{eqn:GLE:K'}
\begin{aligned}
dx(t)&=v(t)dt,\\
m\, dv(t)&=-\gamma v(t)\, dt- \U'(x(t)) \, dt-K(0)x(t)\, dt+ \int_{-\infty}^t\close K'(t-r)x(r)\, dr \, dt \\
&\qquad+  \sqrt{2\gamma} \,dW(t)+ F(t) \, dt.
\end{aligned}
\end{equation}
\begin{proof}[Proof of Proposition~\ref{lem:Q_infinity:equivalent}]  Suppose $\xi_0, \xitilde_0 \in \C_\varrho(-\infty,0]$. Let $\overline{\xi}_0=\xi_0-\xitilde_0$ and observe that~\eqref{eqn:GLE:K'} with $m=\gamma=1$ and initial condition $\widetilde{\xi}_0$ can be expressed as 
\begin{equation} \label{eqn:GLE:int.K'(t-r)x(r):shift.noise}
\begin{aligned}
d \xtil(t)&=\vtil(t) \, dt,\\
d \vtil(t)&=-\vtil(t)\, dt-\U'(\xtil(t))\, dt- K(0) \xtil(t)dt+ \int_{-\infty}^0 K'(t-r) x_0(r)dr dt\\
&\qquad+\int_0^t K'(t-r) \xtil(r)dr dt +F(t) \, dt\\
&\qquad\ +\sqrt{2}dW(t)-\int_{-\infty}^0 K'(t-r) \xbar_0(r)dr dt. 
\end{aligned}
\end{equation}
If the following Novikov condition is satisfied 
\begin{align*}
\E\exp\Big\{\frac{1}{2}\int_0^\infty\close \Big(\int_{-\infty}^0 K'(t-r)\xbar_0(r)dr\Big)^2 dt\Big\}<\infty, 
\end{align*}
then Girsanov's theorem would imply the desired measure equivalence on future paths.  
Since $\xbar_0$ is deterministic, it suffices to show that the above integral is finite. To this end, we note that  since $\xi_0, \widetilde{\xi}_0\in\Cpastrho$, $\xbar_0(\cdot)$ satisfies the growth bound \begin{align*}
\| \xbar_0 \|_{\varrho}:=\sup_{r\leq 0}\frac{|\xbar_0(r)|}{1+|r|^\varrho}<\infty. 
\end{align*}
Using this fact, we estimate as follows: 
\begin{align*}
\int_0^\infty\close \Big(\int_{-\infty}^0 K'(t-r) \xbar_0(r)dr\Big)^2 dt&= \int_0^\infty\bigg( \int_{-\infty}^0 K'(t-r)(1+ |r|^\varrho ) \frac{\xbar_{0}(r)}{1+ |r|^\varrho} \, dr \bigg)^2 \, dt\\
& \leq \| \xbar_0 \|_{\varrho}^2  \int_0^\infty\bigg( \int_{-\infty}^0 K'(t-r)(1+ |r|^\varrho ) \, dr \bigg)^2 \, dt\\
&= \| \xbar_0 \|_{\varrho}^2  \int_0^\infty\bigg( \int_{0}^\infty K'(t+r)(1+ r^\varrho ) \, dr \bigg)^2 \, dt.\end{align*}
For $\epsilon>0$ to be chosen later, recalling by L'Hospital's rule applied to $K(t)/t^{\epsilon-\alpha}$, by Assumption~\ref{cond:K:K(t)<t^-alpha}, we saw that $K'(t) /  t^{\epsilon-\alpha -1}\to 0$ as $t\rightarrow \infty$. Hence, there exist $C>0$ and $t_0 >1$ such that $|K'(t)| \leq C t^{\epsilon-\alpha -1}$ for all $t \geq t_0$.  It then follows that 
\begin{align*}
\int_0^\infty\Big( \int_{0}^\infty &K'(t+r)(1+ r^\varrho ) \, dr \Big)^2 \, dt\\
& = \int_0^{t_0} \Big( \int_{0}^{t_0} K'(t+r)(1+ r^\varrho ) \, dr + \int_{t_0}^{\infty} K'(t+r)(1+ r^\varrho ) \, dr \Big)^2 \, dt\\
&\qquad \qquad+ \int_{t_0}^\infty\Big( \int_{0}^\infty K'(t+r)(1+ r^\varrho ) \, dr \Big)^2 \, dt\\
& \qquad \leq C_1+C_2 \int_0^{t_0} \Big( \int_{t_0}^\infty \frac{1+r^\varrho}{r^{1+\alpha-\epsilon}} \, dr\Big)^2 \, dt  + C_3 \int_{t_0}^\infty\Big( \int_{0}^\infty \frac{1+ r^\varrho }{(t+r)^{1+\alpha-\epsilon}} \, dr \Big)^2 \, dt. 
\end{align*}
Choosing $0<\epsilon<\alpha-\varrho-1/2$, notice that the first integral on the right hand side of the last line above is finite since $\alpha-\epsilon > \varrho$.  For the final integral above, recalling that $t_0>1$ and making the substitution $u= r/t$ produces 
\begin{align*}
\int_{t_0}^\infty\bigg( \int_{0}^\infty \frac{1+ r^\varrho }{(t+r)^{1+\alpha-\epsilon}} \, dr \bigg)^2 \, dt &= \int_{t_0}^\infty \frac{1}{t^{2(\alpha-\epsilon)}} \bigg( \int_{0}^\infty \frac{1+ t^{\varrho} u^\varrho }{(1+u)^{1+\alpha}} \, du \bigg)^2 \, dt\\
& \le  \int_{t_0}^\infty \frac{1}{t^{2(\alpha-\epsilon-\varrho)}}d t \bigg( \int_{0}^\infty \frac{1+  u^\varrho }{(1+u)^{1+\alpha}} \, du \bigg)^2 < \infty,
\end{align*}
since $ \alpha-\epsilon - \varrho>1/2$.  This finishes the proof. 

\end{proof}

We now turn to the proof of Proposition~\ref{lem:Lebesgue-equivalent}. In
order to show equivalence in measures, we aim to
compare~\eqref{eqn:GLE} with a standard, memoryless Langevin
equation. However, because of the memory terms and the nonlinearity
$\U'$, we do not do so directly. Instead, we will consider a truncated
version of~\eqref{eqn:GLE}, which will be useful in
verifying Novikov's condition. More precisely, let $\theta_n \in C^\infty(\rbb, [0,1])$ satisfy $\theta_n(x) = 1$ for all $|x| \leq n$ and $\theta_n(x)= 0$ for $|x| \geq n+1$, and consider the following system with initial path $\xi_0$
\begin{equation} \label{eqn:GLE:truncating:memory}
\begin{aligned}
d\, x(t)&=v(t) dt,\\
d\, v(t)&=-v(t)dt-\U'(x(t))dt+\sqrt{2}dW(t)\\
&\qquad+\theta_n\big(|x(t)|+|v(t)|+|F(t)|\big)\bigg(-\int_{-\infty}^t K(t-r)v(r)dr+F(t) \,\bigg) dt.
\end{aligned}
\end{equation}
In the following auxiliary result, we show that the solution of~\eqref{eqn:GLE:truncating:memory} converges to that of~\eqref{eqn:GLE} as $n$ tends to infinity. 
\begin{lemma} \label{lem:GLE:truncating:memory}
Given an initial condition $\xi_0\in \C(-\infty,0]$ as in~\eqref{eqn:Cpast:int.K'(r).x(r)<infty}, let $\xi^n$ and $\xi$ respectively be the solutions of~\eqref{eqn:GLE:truncating:memory} and~\eqref{eqn:GLE} (with $m=\gamma=1$ in~\eqref{eqn:GLE}) with the same initial history $\xi_0$.  Then, for all $t\ge 0$, 
\begin{equation} \label{eqn:lim:x^R.converge.x}
\lim_{n\to\infty}\E\sup_{0\le r\le t}\bigg\{|x^n(r)-x(r)|+|v^n(r)-v(r)|\bigg\}=0.  
\end{equation}

\end{lemma}

The proof of Lemma~\ref{lem:GLE:truncating:memory} follows a standard comparison argument that will be deferred to the end of this section. Assuming this result, we now establish Proposition~\ref{lem:Lebesgue-equivalent}.
\begin{proof}[Proof of Proposition~\ref{lem:Lebesgue-equivalent}] Let $Q^F_t(\xi_0,\,\cdot\,)$ be the law at time $t$ of $\f^{F,W}_{0,t}\xi_0$ on $\rbb^2$. By stationarity,
\begin{align*}
\pi_t\Law(\xi,\,\cdot\,|F)=\int Q^{F}_t(\xi_0,\,\cdot\,)\nu_{(-\infty,0]}(d\xi_0\times dF).
\end{align*}
It therefore suffices to show that $Q^F_t(\xi_0,\,\cdot\,)$ is equivalent to Lebesgue measure. 

Recalling that $\xi^n=(x^n, v^n)$ denotes the solution of~\eqref{eqn:GLE:truncating:memory}, let $Q^{n,F}_{[0,t]}(\xi_0,\,\cdot\,)$ be the law induced by $\xi^n$ on $C([0,t];\rbb^2)$ and let $Q^{n,F}_t(\xi_0,\,\cdot\,)$ be the marginal of $Q^{n,F}_{[0,t]}(\xi_0,\,\cdot\,)$ at time $t$. We note that 
\begin{align*}
\int_{-\infty}^tK(t-r)v^n(r)dr=\int_{-\infty}^0K(t-r)v_0(r)dr+\int_0^t K(t-r)v^n(r)dr.
\end{align*}
By Assumption~\ref{cond:K:supK(t+s)/K(s)} and the definition of $\theta_n$, the following estimate holds almost surely
\begin{align*}
\theta_n\bigg(|x^n(t)|+|v^n(t)|+|F(t)|\bigg)&\bigg|-\int_{-\infty}^t K(t-r)v(r)dr+F(t)\bigg| \\
&\le \int_{-\infty}^0 K(t-r)|v_0(r)|dr+n\int_0^t K(t-r)dr+n\\
&\qquad\le \widetilde{K}(t)\int_{-\infty}^0 K(-r)|v_0(r)|dr+n\int_0^t K(t-r)dr+n,
\end{align*}
implying the following Novikov-type condition is satisfied
\begin{align*}
&\E \exp\bigg\{\tfrac{1}{2}\int_0^t \theta_n\big(|x^n(r)|+|v^n(r)|+|F(r)|\big)^2\bigg(-\int_{-\infty}^r K(r-\ell)v(\ell)d\ell+F(r)\bigg)^2 dr\bigg\}<\infty.\end{align*}
As a consequence, $Q^{n,F}_{[0,t]}(\xi_0,\,\cdot\,)$ is equivalent to the law $\widetilde{Q}_{[0,t]}(\xi_0(0),\,\cdot\,)$ induced by the solution of the following Langevin equation
\begin{align*}
d\, x(t)&=v(t)dt, & x(0)=x_0(0),\\
d\, v(t)&=-v(t)dt-\U'(x(t))dt+\sqrt{2}dW(t),& v(0)=v_0(0).
\end{align*}
The above system is well-understood. By verifying Hormander's condition, it is clear that $\widetilde{Q}_{t}(\xi_0(0),\,\cdot\,)$ as the marginal law of $\widetilde{Q}_{[0,t]}(\xi_0(0),\,\cdot\,)$ at time $t$ is equivalent to Lebesgue measure \cite{pavliotis2014stochastic}. It follows immediately that $Q^{n,F}_{t}(\xi_0,\,\cdot\,)$ must be too. By taking $n$ to infinity, in  light of Lemma~\ref{lem:GLE:truncating:memory}, $Q^{n,F}_{t}(\xi_0,\,\cdot\,)$ converges to $Q^{F}_{t}(\xi_0,\,\cdot\,)$, which preserves measure equivalence. The proof is thus complete.
\end{proof}

We finally give the proof of Lemma~\ref{lem:GLE:truncating:memory} whose proof is somewhat standard. 
\begin{proof}[Proof of Lemma~\ref{lem:GLE:truncating:memory}]
We first note that by adapting the energy estimate as in the proof of Proposition~\ref{prop:well-posedness} to~\eqref{eqn:GLE:truncating:memory}, we have the following uniform bound in $n$
\begin{align} 
&\E\sup_{r \in [0, t]} H(x_n(r), v_n(r))\notag \\
&\le \Big(H(x_0(0), v_0(0))+\Big(\int_{-T}^0 K(-w) |v_0(w)| \, dw\Big)^2+ \E \sup_{r\in [0,t]}|F(r)|^2+1   \Big)e^{c(t)}.\label{ineq:E.sup_(0<r<t)H(x_n(r), v_n(r)):memory}
\end{align}
Now consider the stopping time $\tau_n$ given by
\begin{align*}
\tau_n =\inf\{t\geq 0: |x(t)|+|v(t)|+ |F(t)| \geq n\}.
\end{align*}
It is clear that $\xi(r)=\xi^n(r)$ for all $0\le r\le \tau_n$. Using Holder's inequality and recalling $\delta\in(0,1)$ as in Assumption~\ref{cond:Phi}:
\begin{align*}
&\E\sup_{0\le r\le t}|x^n(r)-x(r)|+|v^n(r)-v(r)|\\
&=\E\Big(\sup_{0\le r\le t}|x^n(r)-x(r)|+|v^n(r)-v(r)|1\{\tau_n\le t\}\Big)\\
&\le c\Big(\E\sup_{0\le r\le t} |x^n(r)|^{1+\delta} +|v^n(r)|^{1+\delta} +\E\sup_{0\le r\le t} |x(r)|^{1+\delta} +|v(r)|^{1+\delta} \Big)^{1/(1+\delta )}\Big(\P(\tau_n\le t)\Big)^{\delta/(1+\delta)}. 
\end{align*}
We invoke the energy estimates~\eqref{ineq:energy-estimate} and~\eqref{ineq:E.sup_(0<r<t)H(x_n(r), v_n(r)):memory} and recall that the nonlinear potential $U$ dominates $|x|^{1+\delta}$, cf. Assumption~\ref{cond:Phi}, to see that 
\begin{align*}
&\E\sup_{0\le r\le t} |x^n(r)|^{1+\delta}+|v^n(r)|^{1+\delta}+\E\sup_{0\le r\le t} |x(r)|^{1+\delta}+|v(r)|^{1+\delta}\\
&\le c \Big(\E\sup_{0\le r\le t} |x^n(r)|^{1+\delta}+|v^n(r)|^{2}+\E\sup_{0\le r\le t} |x(r)|^{1+\delta}+|v(r)|^{2}+1\Big)\le e^{c(t,\xi_0,F)},
\end{align*}
where $c(t,\xi_0,F)>0$ is a constant independent of $n$. Also, by Chebyshev's inequality and Lemma~\ref{lem:sup.F_t^2},
\begin{align*}
\P(\tau_n\le t)&= \P(\sup_{0\le r\le t}|x(r)|+|v(r)|+ F(r) \geq n)\\&\le \frac{1}{n}\Big(\E\sup_{0\le r\le t} |x(r)|+|v(r)|+\sup_{0\le r\le t} |F(r)|\Big)\le \frac{e^{c(t,\xi_0,F)}}{n}.
\end{align*}
Altogether, we arrive at the bound
\begin{align*}
\E\sup_{0\le r\le t}|x^n(r)-x(r)|+|v^n(r)-v(r)|&\le \frac{e^{c(t,\xi_0,F)}}{\sqrt{n}},
\end{align*}
which converges to zero as $n$ tends to infinity. This finishes the proof.
\end{proof}

\section{Existence of an invariant measure} 
\label{sec:existence}
In this section, we assume that the memory kernel $K$ is of the form 
\begin{align}
K(t) = \sum_{\ell \geq 1}\int_0^\infty J_\ell(s+t) J_\ell(s) \, ds,
\end{align}
where the functions $J_\ell$, $\ell \geq 1$, are as in Remark~\ref{rem:specialform}.  In this case, we will see here that we can construct an explicit stationary measure for the Markov flow on $\C_\varrho((-\infty, 0])$  by pulling back a known invariant measure for an augmented version of~\eqref{eqn:GLE}.

Introducing the auxiliary variable $z_k(t)$ given by
\begin{align} \label{eqn:z_k}
z_\ell(t)=\sqrt{c_\ell}\int_{-\infty}^t  e^{-\lambda_\ell(t-r)}v(r)dr dt+\sqrt{2\lambda_\ell }\int_{-\infty}^t \close e^{-\lambda_\ell(t-r)}dB^{(\ell)}(r)dt,
\end{align}
we find that equation~\eqref{eqn:GLE} can be expressed as 
\begin{equation}\label{eqn:GLE-Markov:xvz}
\begin{aligned}
d\, x(t) &= v(t)\, d t, \\
 d\, v(t)&=- v(t) \, dt-\U'(x(t)) \, dt-\textstyle{\sum}_{\ell\geq 1} \sqrt{c_\ell} z_\ell(t)\,dt+\sqrt{2} \,dW(t) ,\\
d\, z_\ell(t)&=-\lambda_\ell z_\ell(t) \, dt+ \sqrt{c_\ell}v(t) \, dt+\sqrt{2\lambda_\ell}\, dB^{(\ell)}(t),  \,\,\, \ell\geq 1.
\end{aligned}	
\end{equation}
In this setting, the relationship between  the system above and the original equation~\eqref{eqn:GLE} must account for a specific initial condition in the past.  For now, however, we view this system as a Markovian dynamics started from a given initial condition on the phase space $\H_{-s}$ where 
\begin{equation} \label{eqn:H_-s}
\H_{-s}:=\big\{X=(x,v,z_1,\dots):\|X\|_{-s}^2:=x^2+v^2+\textstyle{\sum}_{\ell\geq 1}\ell^{-2s}z_\ell^2<\infty\big\}.  
\end{equation}
In the above, the real parameter $s$ is such that 
\begin{align}
1<2s < \alpha \beta,
\end{align}
and $\alpha, \beta >0$ are as in Remark~\ref{rem:specialform}.  Under these hypotheses, the system~\eqref{eqn:GLE-Markov:xvz} is well-posed on $\H_{-s}$, and the probability measure on $\H_{-s}$ given by 
\begin{align}
 \label{eqn:mu_xvz}
\mu \propto \pi\times\prod_{\ell\ge 1} \nu_\ell,
\end{align}
where $\pi$ is the Boltzmann-Gibbs measure in $(x,v)$ as in~\eqref{eqn:marginal-x-v} and $\{\nu_k\}_{k\ge 1}$ are independent copies of the standard normal distribution $N(0,1)$ on $\rbb$, is an invariant probability measure for the Markov process~\eqref{eqn:GLE-Markov:xvz}~\cite[Theorem 7]{glatt2018generalized}.  

\subsection{The induced measure on path space}

Consider an arbitrary collection of real numbers 
$$t_1 \leq t_2 \leq \cdots \leq t_n , $$
and a collection of Borel sets $A_1,\dots,A_n\subset \Hs$. If $X_{t_1}(\,\cdot\,)$ denotes the solution of~\eqref{eqn:GLE-Markov:xvz} distributed as $\mu$ at time $t_1$, we define $\widehat{\mu}_{t_1,\dots,t_n}$ on the cylinder set $A_1\times\cdots\times A_n$ by
\begin{align} \label{eq:mu_*(A_1...A_n)}
\widehat{\mu}_{t_1,\dots,t_n}(A_1 \times \cdots \times A_n)= \P\{ X_{t_1}(t_1) \in A_1, \ldots, X_{t_1}(t_n) \in A_n \}.  
\end{align}   
Since $\mu$ is invariant for the Markov process~\eqref{eqn:GLE-Markov:xvz}, it can be shown by Kolmogorov's extension theorem (by taking a continuous version of the process $X$ solving~\eqref{eqn:GLE-Markov:xvz}) that the family $$\{\widehat{\mu}_{t_1,\dots,t_n}\,: \, n \in \nbb,\, t_1 \leq t_2 \leq \ldots \leq t_n\},$$ is consistent, and hence induces a stationary measure, denoted by $\widehat{\mu}$, on Borel subsets of $C(\rbb,\Hs)$ whose finite-dimensional distributions are as in~\eqref{eq:mu_*(A_1...A_n)}.  Let $\xi_*=(x_*,v_*)$ denote the projection of the corresponding stationary process on $C(\rbb, \H_{-s})$ onto $C((-\infty, 0], \rbb^2)$.  By definition, it follows that $\psi_t^W(\xi_*, B)$ is stationary on $\S(-\infty, 0]$ given by~\eqref{form:S(-infty,t]}.  Let $\mu_*$ denote its corresponding distribution, which in particular is invariant for the Markov semigroup $P_t$ as in~\eqref{eq:MarkovOnC}.

We will next show that $\mu_*$ concentrates on a path space with moderate growth, {thereby finishing the proof of Theorem~\ref{thm:existence}}.

\begin{lemma} \label{lem:mu_p.concentrate.CpastHsrho} Let $\mu_*$ be the probability measure in $\S(-\infty, 0]$ constructed above. Then, for every $\varrho>0$, $$\mu_*(\C_{\varrho}(-\infty, 0])=1,$$
where $\C_{\varrho}(-\infty, 0]$ is as in \eqref{eqn:CpastR^2}.
\end{lemma} 

\begin{proof}
By Borel-Cantelli, it suffices to prove that 
\begin{align*}
\sum_{n\geq 1}\widehat{\mu}\bigg\{ X(\cdot)=(\xi, z_1, z_2, \ldots)\in C(\rbb, \H_{-s})\, :\,\sup_{-n\leq r\leq -n+1}|x(r)|>(n+1)^{\varrho}\bigg\}<\infty.
\end{align*}
By invariance
\begin{align*}
\widehat{\mu}\bigg\{ X(\cdot):\sup_{-n\leq r \leq-n+1}|x(r)|>(n+1)^{\varrho}\bigg\}&=\widehat{\mu}\bigg\{\sup_{0\leq t\leq 1}|x_{0}(t)|>(n+1)^{\varrho} \bigg\},
\end{align*}
where  $X_0(\, \cdot \,)= (\xi_0(\, \cdot \, ), z_1( \, \cdot \, ), \ldots)$ denotes the solution of~\eqref{eqn:GLE-Markov:xvz} with initial distribution $\mu$ at time $0$. To estimate the righthand side above, we apply It\^{o}'s formula to the Hamiltonian $H(\xi)=H(x,v)= \frac{1}{2}v^2+ \U(x)$, and obtain for $t\ge 0$
\begin{equation} \label{eq:Ito:Psi(X)}
\begin{aligned}
d H(\xi(t))&= -v(t)^2dt+1dt+v(t)dW(t)-\sum_{\ell\geq 1}\sqrt{c_\ell}z_\ell(t)v(t)dt.
\end{aligned}
\end{equation}
The cross terms involving $z_\ell(t)$ and $v(t)$ can be bounded from above by
\begin{align*}
\sqrt{c_\ell}|z_\ell(t)v(t)|\le  C \ell^{-2s}z_\ell(t)^2+\frac{c_\ell\ell^{2s}}{C} v(t)^2,
\end{align*} 
where $C>0$ is large enough such that $C \textstyle{\sum_{\ell \geq 1}} c_\ell \ell^{2s}<1$.  
Integrating~\eqref{eq:Ito:Psi(X)} on $[0,t]$, $t\le 1$ using the estimates above then produces
\begin{align*}
H(\xi(t))&\leq H(\xi(0))+ 1+ C\int_0^1 \sum_{\ell \geq 1}  \ell^{-2s}z_\ell(s)^2\, ds+\sup_{0\leq t\leq 1}\int_0^tv(r)dW(r). 
\end{align*}
Fixing $\varepsilon\in(1/2,s)$ and recalling Assumption~\ref{cond:Phi}, namely, $U(x)$ dominates $|x|^{1+\delta}$, $\delta\in(0,1)$, we have the following chain of implications
\begin{align*}
&\Big\{\sup_{0\leq t\leq 1}|x_0(t)|\geq (n+1)^{\varrho}\Big\}\subset\Big\{\sup_{0\le t\le 1} U(x_0(t))\ge c(n+1)^{(1+\delta)\varrho}\Big\}\\
&\subset \Big\{H(\xi(0))+1+ C\textstyle{\int_0^1 \sum_{\ell \geq 1}  \ell^{-2s}z_\ell(s)^2\, ds}+\textstyle{\sup_{0\leq t\leq 1}\int_0^tv(r)dW(r)}\geq c(n+1)^{(1+\delta)\varrho}\Big\}.\\
&\subset  \Big\{\U(x(0))\geq c(n+1)^{(1+\delta)\varrho}\}\cup\{\tfrac{1}{2}v(0)^2+1+\textstyle{\sup_{0\leq t\leq 1}\int_0^tv(r)dW(r)}\geq c(n+1)^{(1+\delta)\varrho}\Big\}\\
&\qquad \bigcup_{\ell \geq 1}  \Big\{ C\textstyle{\int_0^1   \ell^{-2s}z_\ell(s)^2\, ds} \geq c (n+1)^{(1+\delta)\varrho}\ell^{-2\varepsilon} \Big\}=I_{x}\cup I_v\cup_{\ell\ge 1}I_\ell.\end{align*}
We are left to estimate each of the above events. For $p>0$, using Chebychev's inequality, we estimate $I_x$ as follows: 
\begin{align*}
\widehat{\mu}(I_x)=\mu\{I_x\}&\leq \frac{c}{(n+1)^{(1+\delta)\varrho p}} \int_{\Hs}\close \U(x)^p\mu(dX)\\
&=\frac{c}{(n+1)^{(1+\delta)\varrho p}}\int_\rbb \U(x)^p e^{-\U(x)}dx\leq \frac{c}{(n+1)^{(1+\delta)\varrho p}}.
\end{align*}
For $I_v$, we first employ Burkholder's inequality to see that
\begin{align*}
\E\sup_{0\leq t\leq 1}\bigg(\int_0^tv(r)dB_0(r)\bigg)^{2p}\leq c\int_0^1\E|v(r)|^{2p}dr.
\end{align*}
Using the fact that $\mu$ is invariant for $X_0(t)$, we estimate $I_v$ 
\begin{align*}
\widehat{\mu}(I_v)
&\leq \frac{c}{(n+1)^{(1+\delta)\varrho p}} \bigg[\int_{\Hs}\close v^{2p}\mu(dX)+\int_0^1\E\int_{\Hs}\close v(r)^{2p}\mu(dX)dr\bigg]\\
&=\frac{c}{(n+1)^{(1+\delta)\varrho p}}\int_{\Hs}\close 2v^{2p}\mu(dX) \\
&=\frac{c}{(n+1)^{(1+\delta)\varrho p}}\int_{\rbb}2v^{2p}e^{-v^2/2}dv\leq \frac{c}{(n+1)^{(1+\delta)\varrho p}}.
\end{align*}
Likewise, for $I_\ell$ we find that  
\begin{align*}
\widehat{\mu} (I_\ell)
&\leq  \frac{c \ell^{2(\varepsilon-s) p} }{(n+1)^{(1+\delta)\varrho p}}\int_0^1\E\int_{\Hs}| z_\ell^2(r)|^{p}\mu(dX)dr\\
&=  \frac{c \ell^{2(\varepsilon-s) p} }{(n+1)^{(1+\delta)\varrho p}}\int_\rbb |z|^{2p}e^{-z^2/2}dz.
\end{align*}
We now collect everything and note that $\varepsilon\in(1/2,s)$ to arrive at
\begin{align*}
\widehat{\mu}\Big\{\sup_{0\leq t\leq 1}|x(t)|\geq (n+1)^{\varrho}\Big\}\leq \frac{c}{(n+1)^{(1+\delta)\varrho p}}\big[1+\sum_{\ell\geq 1} \ell^{2(\varepsilon-s) p}\big]\le \frac{c}{(n+1)^{(1+\delta)\varrho p}},
\end{align*}
which holds for $p$ sufficiently large, e.g., $2(s-\varepsilon)p>1$. Furthermore, we emphasize that the above constant $c$ is independent of $n$. It follows that
\begin{align*}
\sum_{n\geq 1}\widehat{\mu}\bigg\{ X(\cdot)=(\xi, z_1, z_2, \ldots)\in C(\rbb, \H_{-s})\, :\,\sup_{-n\leq r\leq -n+1}|x(r)|>(n+1)^{\varrho}\bigg\}\leq \sum_{n\geq 1}\frac{c}{(n+1)^{(1+\delta)\varrho p}},
\end{align*}
which is summable as long as $p$ is chosen such that $(1+\delta)\varrho p>1$. The proof is thus complete.
 
\end{proof}

\section*{acknowlegement} The authors graciously acknowledge support
from the Department of Mathematics at Duke University and the
Department of Mathematics at Iowa State University.  DPH was supported
in part by NSF Grants DMS-1612898  and DMS-1855504. JCM thanks the NSF
 for its partial suport through the grant DMS-1613337. The authors would like to thank Gustavo Didier for helpful discussions in the development of this work. The authors also would like to thank the anonymous reviewer for their valuable comments and suggestions.

\appendix

\section{Well-posedness} \label{sec:well-posed}

In this section, we show that equation~\eqref{eqn:GLE} is well-posed as stated in Proposition~\ref{prop:well-posedness}. We first construct strong, i.e., pathwise, solutions. Then, the existence and uniqueness of weak solutions simply follow by using classical arguments \cite[Chapter 5]{oksendal2003stochastic}.

First, fixing $T>0$ we consider a slightly different approximating equation
\begin{align}
\label{eqn:GLET}
dx(t)&= v(t) \, dt, \\ 
dv(t)&= -v(t) \, dt - U'(x(t)) \, dt - \int_{-T}^t K(t-s) v(s) \, ds \, dt +  \sqrt{2} \, dW(t) + F(t) \, dt, \nonumber 
\end{align}
where we have truncated the memory term in~\eqref{eqn:GLE} at time $-T$.  Following the standard iteration procedure for standard SDEs with globally Lipschitz coefficients~\cite{bass1998diffusions, oksendal2003stochastic}, we can obtain the well-posedness of relation~\eqref{eqn:GLET} assuming that $U'$ is globally Lipschitz:    
\begin{lemma}
\label{lem:approxT}
Fix $T>0$, $K\in C(\rbb)$ and suppose that $U'$ is globally Lipschitz. Then for all $\xi_0=(x_0, v_0) \in \C(-\infty, 0]$, there exists a unique continuous adapted solution $\xi^T(t)=(x^T(t), v^T(t))$ of equation~\ref{eqn:GLET} for all times $t\geq 0$ with $\xi^T(0)=(x_0(0), v_0(0))$.     
\end{lemma}  
In order to remove the globally Lipschitz hypothesis in Lemma~\ref{lem:approxT}, we use an energy estimate to show absence of explosion under the assumption that $U' \in C^1(\rbb)$ with $U' \rightarrow \infty$ as $|x|\rightarrow \infty$.

 \begin{lemma}
\label{lem:noexpT}
Fix $T>0$, $K\in C(\rbb)$ and suppose Assumption~\ref{cond:K:supK(t+s)/K(s)} holds.  Furthermore, suppose that $U'$ in equation~\eqref{eqn:GLET} satisfies $\U' \rightarrow \infty$ as $|x|\rightarrow \infty$.  Then for all $\xi_0=(x_0, v_0) \in \C(-\infty,0]$, there exists a unique continuous solution $\xi^T(t)=(x^T(t), v^T(t))$ of equation~\eqref{eqn:GLET} for all times $t\geq 0$ with $\xi^T(0)=(x_0(0), v_0(0))$.
\end{lemma} 

\begin{proof}
Recalling $\theta_n$ as in~\eqref{eqn:GLE:truncating:memory}, let $H_n(x,v)= \frac{1}{2}v^2+ \U(x)\theta_n(x)$.  Define $U_n:\rbb\rightarrow \rbb$ by $U_n(x) = \U(x) \theta_n(x)$ and note that the system~\eqref{eqn:GLET} with $U'$ replaced by $U_n'$ has unique solutions $(x_n(t), v_n(t))$ as in Lemma~\ref{lem:approxT} with $(x_n(0), v_n(0))=\xi(0) \in \rbb^2$.  Furthermore, these solutions agree with the solutions of equation~\eqref{eqn:GLET} for all times $t<\sigma_n:= \inf\{ t\geq 0 \, : \, H(x_n(t), v_n(t))\geq n \}$ where $H$ is the Hamiltonian. Now, fix $t>0$ and note that It\^{o}'s formula implies     
\begin{align*}
\E\sup_{r \in [0, t]} H_n(x_n(r), v_n(r))&\leq  H(\xi(0))+\E \sup_{r\in [0, t]} \int_0^r \bigg\{|v_n(u)| \int_{-T}^u K(u-w) |v_n(w)|\, dw\bigg\} \,du \\
&\qquad + \sqrt{2} \E \sup_{r\in [0,t]} \bigg|\int_0^rv_n(u) dW(u)\bigg|+ \E \sup_{r\in [0,t]} \bigg|\int_0^r v_n(u) F(u) \, du\bigg|\\
&=: H(\xi(0))+(I)_t +(II)_t+(III)_t.\end{align*}  
For the term $(I)_t$, we note that Assumption~\ref{cond:K:supK(t+s)/K(s)} gives
\begin{align*}
& \int_0^r \bigg\{|v_n(u)| \int_{-T}^u K(u-w) |v_n(w)|\, dw\bigg\} \,du\\
&=   \int_0^r \bigg\{|v_n(u)| \int_{-T}^0 \frac{K(u-w)}{K(-w)} K(-w) |v_n(w)|\, dw\bigg\} \, du + \int_0^r \bigg\{|v_n(u)| \int_{0}^u K(u-w) |v_n(w)|\, dw\bigg\} \,du \\
&\leq \int_0^r |v_n(u)| \Ktilde(u) \int_{-T}^0 K(-w) |v_0(w)| \, dw \, du + \int_0^r \sup_{w\in [0, u]} |v_n(w)|^2 \int_0^u K(u-w) \, dw \, du.  
\end{align*}   
Hence we can estimate $(I)_t$ as
\begin{align*}
(I)_t \leq  \int_0^t c_1(r) \E \sup_{s \in [0, r]} |v_n(s)|^2 \, dr  + c_2(t)\bigg(\int_{-T}^0 K(-w) |v_0(w)| \, dw\bigg)^2,
\end{align*}
for some continuous functions $c_i$ on $[0,t]$.  For the term $(II)_t$, note that Doob's Maximal Inequality implies
\begin{align*}
(II)_t=\sqrt{2} \E \sup_{r\in [0,t]} \bigg|\int_0^rv_n(u) dW(u)\bigg|\leq c\bigg(1+\int_0^t\E \sup_{u\in [0,r]}|v_n(u)|^2dr\bigg). \end{align*}
Concerning $(III)_t$, we use Young's inequality for products to obtain
\begin{align*}
(III)_t\le \tfrac{1}{2}\int_0^t \E \sup_{u\in [0,r]}|v_n(u)|^2dr+\tfrac{1}{2}t\E \sup_{r\in [0,t]}|F(r)|^2.
\end{align*} 
We collect the estimates above to arrive at the bound
\begin{align*}
\E\sup_{r \in [0, t]} H_n(x_n(r), v_n(r))&\leq  H_n(\xi(0))+c_1(t)\int_0^t \E\sup_{u \in [0, r]} H_n(x_n(u), v_n(u))dr\\
&\qquad+c_2(t) \bigg(\int_{-T}^0 K(-w) |v_0(w)| \, dw\bigg)^2+\tfrac{1}{2}t\E \sup_{r\in [0,t]}|F(r)|^2+c_3(t),
\end{align*}
whence using Gr\"{o}nwall's inequality and the Monotone Convergence Theorem
\begin{align} \label{ineq:E.sup_(0<r<t)H(x_n(r), v_n(r))}
&\E\sup_{r \in [0, t]} H(x_n(r), v_n(r))\\
&\le \bigg(H(x_0(0), v_0(0))+\bigg(\int_{-T}^0 K(-w) |v_0(w)| \, dw\bigg)^2+ \E \sup_{r\in [0,t]}|F(r)|^2+1   \bigg)e^{c(t)},\nonumber                                                                                                                                                                                                                                                                                                                                                                                                                                                                                                                                                                                                                                                                                                                                                                                                                                                                                                                                                                                                                                                                                                                                                                                                                                                                                                                                                                                                                                                                                                                                                                                                                                                                                                                                                                                                                                                                                                                                                                                                                                                                                                                                                                                                                                                                                                                                                                                                                                                                                                                                                                                                                                                                                                                                                                                                                                                                                                                                                                                                                                                                                                                                                                                                                                                                                                                                                                                                                                                                                                                                                                                                                                                                                                                                                                                                                                                                                                                                                                                                                                                                                                                                                                                                                                                                                                                                                                                                                                                                                                                                                                                                                                                                                                                                                                                                                                                                                                                                                                                                                                                                                                                                                                                                                                                                                                                                                                                                                                                                                                                                      \end{align}

Turning back to $\sigma_n$, we note that
\begin{align} \label{ineq:P(sigma_n<t)<H(x,v)}
\E\sup_{r \in [0, t]} H_n(x_n(r), v_n(r))&\ge \E\bigg[\sup_{r \in [0, t]} H_n(x_n(r), v_n(r))\cdot 1\{\sigma_{n-1} < t\}\bigg]\\
&\ge (n-1)\P(\sigma_{n-1}<t),\nonumber
\end{align}
which together with~\eqref{ineq:E.sup_(0<r<t)H(x_n(r), v_n(r))} yields
$
\P(\sigma_n<t)\le \frac{1}{n}c(t).
$
By taking $n$ to infinity, we immediately obtain $\P(\sigma_\infty<t)=0$ for any $t \geq 0$.  Hence $\P(\sigma_\infty = \infty)=1$, finishing the proof.  
\end{proof}

Our next goal is to allow the memory to depend on the infinite past by carefully passing $T$ to infinity in~\eqref{eqn:GLET}.  
\begin{lemma} \label{lem:well-posed:pathwise}
Let $T>0$, $\xi_0=(x_0, v_0) \in \C(-\infty,0]$ and suppose $K$ satisfies Assumption~\ref{cond:K:supK(t+s)/K(s)}.  Suppose $\U \in C^1(\rbb)$ is such that $\U(x)\rightarrow \infty$ as $|x|\rightarrow \infty$, and let $\xi^T(t)=(x^T(t), v^T(t))$ denote the solution of equation~\eqref{eqn:GLET} with $\xi^T(0)= \xi_0(0)$.  Then for any $t>0$, the solution $\xi^T$ converges as $T\rightarrow \infty$ to $\xi$ in $C([0, t], \rbb^2)$.  Furthermore, $\xi$ is the unique pathwise solution of~\eqref{eqn:GLE} with $\xi(0)=\xi_0(0).$    
\end{lemma}

\begin{proof}
Let $t>0$.  Uniqueness of solutions and the fact that the presumed limit solves~\eqref{eqn:GLE} both follow almost immediately once we show that an appropriate approximating sequence is Cauchy in $C([0, t]; \rbb^2)$.  To be more precise, for $T_1 \geq T_2 >0$, let $ \xi^{T_1}_n=(x_n^{T_1},v_n^{T_1})$ and $\xi^{T_1}_n=(x_n^{T_2},v_n^{T_2})$ respectively be the solutions of~\eqref{eqn:GLET} with $U'(x)$ being replaced by $U_n'(x)$ where $U_n(x)=\U(x) \theta_n(x)$ as in the proof of Lemma~\ref{lem:noexpT}.  For simplicity, let $\xibar_n=\xi^{T_1}_n-\xi^{T_2}_n= (\xbar_n, \vbar_n)$ and observe that
\begin{align*}
|\xbar_n(t)|+|\vbar_n(t)|&\le 2\int_0^t | \vbar_n(r)|dr+\int_0^t\big|U_n(x_n^{T_1}(r))-U_n(x_n^{T_2}(r)\big|dr\\
&\qquad+\int_0^t\int_{-T_1}^{-T_2}K(r-u)|v_0(u)|dudr+\int_0^t\int_0^r K(r-u)|\vbar_n(u)|dudr.
\end{align*}
Note that by Assumption~\ref{cond:K:supK(t+s)/K(s)},
\begin{align*}
\int_0^t\int_{-T_1}^{-T_2}K(r-u)|v_0(u)|dudr&= \int_0^t\int_{-T_1}^{-T_2}\frac{K(r-u)}{K(u)}K(u)|v_0(u)|dudr\\
& \le \int_0^t \Ktilde(r)dr\cdot \int_{-T_1}^{-T_2}K(u)|v_0(u)|du.
\end{align*}
Using the fact that $U_n$ is Lipschitz we then obtain 
\begin{align*}
\sup_{0\le r\le t}|\xbar_n(r)|+|\vbar_n(r)|&\le c(t,n)\int_0^t \sup_{0\le u\le r}|\xbar_n(u)|+|\vbar_n(u)|dr+c(t)\int_{-T_1}^{-T_2}K(u)|v_0(u)|du.
\end{align*}
Thus Gr\"{o}nwall's inequality gives
\begin{align} \label{ineq:sup_(0<r<t)|xbar_n(r)|+|vbar_n(r)|}
\sup_{0\le r\le t}|\xbar_n(r)|+|\vbar_n(r)|\le e^{c(n,t)}\int_{-T_1}^{-T_2}K(u)|v_0(u)|du.
\end{align}
Next, let $\sigma^{T_1}_n$ and $\sigma^{T_2}_n$ respectively denote the stopping times associated with $\xi^{T_1}_n(t)$ and $\xi^{T_2}_n(t)$ as in the proof of Lemma~\ref{lem:noexpT}. Setting $\xibar(t)=\xi^{T_1}(t)-\xi^{T_2}(t)$ we find that 
\begin{align*}
\E\bigg[\sup_{0\le r\le t}|\xbar(r)|+|\vbar(r)|\bigg]&\le\E\bigg[1\{\sigma^{T_1}_n\wedge\sigma^{T_2}_n\ge t\}\sup_{0\le r\le t}|\xbar(r)|+|\vbar(r)| \bigg]\\
&\qquad+\E\bigg[1\{\sigma^{T_1}_n\le t\}\sup_{0\le r\le t}|\xbar(r)|+|\vbar(r)| \bigg]\\
&\qquad+\E\bigg[1\{\sigma^{T_2}_n\le t\}\sup_{0\le r\le t}|\xbar(r)|+|\vbar(r)| \bigg]\\
&=(I)_t+(II)_t+(III)_t.
\end{align*}
In view of~\eqref{ineq:sup_(0<r<t)|xbar_n(r)|+|vbar_n(r)|}, we have
\begin{align*}
(I)_t\le \E\sup_{0\le r\le t}|\xbar_n(r)|+|\vbar_n(r)|\le e^{c(n,t)}\int_{-T_1}^{-T_2}K(u)|v_0(u)|du.
\end{align*}
Concerning $(II)_t$, we use Holder's inequality and Assumption~\ref{cond:Phi} to infer the bound
\begin{align*}
(II)_t&\le \bigg(\E\bigg[\sup_{0\le r\le t}|\xbar(r)|+|\vbar(r)|\bigg]^{1+\delta}\bigg)^\frac{1}{1+\delta}\Big(\P\Big(\sigma^{T_1}_n\le t\Big)\Big)^{\frac{\delta}{1+\delta}}\\
&\le c\bigg(1+\E\sup_{0\le r\le t}|\xbar(r)|^{1+\delta}+|\vbar(r)|^2\bigg)^\frac{1}{1+\delta}\Big(\P\Big(\sigma^{T_1}_n\le t\Big)\Big)^\frac{\delta}{1+\delta}\\
&\le c\bigg(1+\E\Big[\sup_{0\le r\le t}H\big(x^{T_1}(r),v^{T_1}(r)\big)+H\big(x^{T_2}(r),v^{T_2}(r)\big)\Big]\bigg)^\frac{1}{1+\delta}\Big(\P\Big(\sigma^{T^1}_n\le t\Big)\Big)^\frac{\delta}{1+\delta}\\
&\le  c\bigg(1+\E\Big[\sup_{0\le r\le t}H\big(x^{T_1}(r),v^{T_1}(r)\big)+H\big(x^{T_2}(r),v^{T_2}(r)\big)\Big]\bigg)^\frac{1}{1+\delta}\cdot \frac{c(t)}{n^{\delta/(1+\delta)}}\le \frac{c(t)}{n^{\delta/(1+\delta)}}.
\end{align*}
In the above estimate, we employed~\eqref{ineq:P(sigma_n<t)<H(x,v)} together with~\eqref{ineq:E.sup_(0<r<t)H(x_n(r), v_n(r))}. Likewise,  \begin{align*}
(III)_t&\le \frac{c(t)}{n^{\delta/(1+\delta)}}.
\end{align*} 
Altogether, we arrive at the bound
\begin{align*}
\E\sup_{0\le r\le t}|\xbar(r)|+|\vbar(r)|&\le e^{c(n,t)}\int_{-T_1}^{-T_2}K(u)|v_0(u)|du+\frac{1}{n^{\delta/(1+\delta)}}c(t).
\end{align*}
Thanks to the assumption that $\xi_0\in\C(-\infty,0]$, it is now clear that $\{\xi^{T}\}$ is a Cauchy sequence in $C([0,t];\rbb^2)$ by first taking $n$ sufficiently large and then sending $T_1$ and $T_2$ to infinity. As a consequence, there exists a solution $\xi$ for~\eqref{eqn:GLE} with the initial condition $\xi_0\in\C$. 

Turning to the uniqueness of $\xi$, it suffices to show that if $\widetilde{\xi}$ solves~\eqref{eqn:GLE} with the same initial path $\xi_0$, then $\xi$ and $\xitilde$ must agree a.s. in $[0,t]$. To see this, consider the stopping times $\sigma_n$ and $\sigmatilde_n$ associated with $\xi$ and $\xitilde$ respectively. Similarly to the above existence part, denoting $\xihat=\xi-\xitilde$, we observe that for $0\le t\le \sigma_n\wedge\sigmatilde_n$, $\xi$ and $\xitilde$ both solve equation~\eqref{eqn:GLE} with $U'$ being replaced by $U_n'(x)$. So that, for $0\le t\le \sigma_n\wedge\sigmatilde_n$, $\xihat$ satisfies $\xihat(0)=0$ and
\begin{align*}
\frac{d}{dt}\xhat(t)&=\vhat(t),\\
\frac{d}{dt}\xhat(t)&=-\vhat(t)-\big[U_n'(x(t))-U_n'(\xtil(t))\big]-\int_0^t K(t-r)\vhat(r)dr.
\end{align*}
Since the nonlinear term is Lipschitz, by Gronwall's inequality, we immediately obtain
\begin{align*}
\E \bigg[1\{\sigma_n\wedge\sigmatilde_n\ge t\}\sup_{0\le r\le t} |\xhat(r)|+|\vhat(r)|\bigg]=0.
\end{align*}
On the other hand, similar to the estimate of $(II)_t$ above, we also have the bound
 \begin{align*}
 &\E \bigg[\Big(1\{\sigma_n\le  t\}+1\{\sigmatilde_n\le  t\}\Big)\sup_{0\le r\le t} |\xhat(r)|+|\vhat(r)|\bigg]\\
 &\le c\bigg(1+\E\Big[\sup_{0\le r\le t}H\big(x(r),v(r)\big)+H\big(\xtil(r),\vtil(r)\big)\Big]\bigg)^{1/(1+\delta)}\cdot \frac{1}{n^{\delta/(1+\delta)}}c(t)\le \frac{1}{n^{\delta/(1+\delta)}}c(t).
 \end{align*}
By taking $n$ large, we observe that $\E \sup_{0\le r\le t} |\xhat(r)|+|\vhat(r)|$ is arbitrarily small, forcing \begin{align*}
\E \sup_{0\le r\le t} |\xhat(r)|+|\vhat(r)|=0,
\end{align*} holds true. The proof is thus complete.
\end{proof}

Given the strong solutions constructed above, we are now ready to give the proof of Proposition~\ref{prop:well-posedness}. The argument is relatively short and can be found in previous works (see, for example, \cite{oksendal2003stochastic}).

\begin{proof}[Proof of Proposition~\ref{prop:well-posedness}]
The existence of weak solution is clear since we already constructed strong solutions as in Lemma~\ref{lem:well-posed:pathwise}. It remains to show weak uniqueness.

Suppose $(\xi,F,W)$ and $(\xitilde,\Ftilde,\Wtilde)$ are two weak solutions as in Definition~\ref{def:initialSolution} on the interval $[t_0,t]$ with the same initial condition $\xi_0$. By the uniqueness of strong solutions, we may consider $\xi$ and $\xitilde$ as the unique path-wise solutions given $(F,W)$ and $(\Ftilde,\Wtilde)$, respectively. To see that $\xi$ and $\xitilde$ have the same law, we recall the construction of $\xi$ starting from system~\eqref{eqn:GLET} with $U'$ being Lipschitz. Then, it is clear that the processes $\xi^T$ and $\xitilde^{\,T}$ as in Lemma~\ref{lem:approxT} must agree in distribution \cite[Lemma 5.3.1]{oksendal2003stochastic}. In view of Lemma~\ref{lem:noexpT}, this property also holds for general $U$ satisfying Assumption~\ref{cond:Phi}. Finally, since $\xi^T$ and $\xitilde^{\,T}$ respectively converge to $\xi$ and $\xitilde$ on $C([t_0,t];\rbb^2)$ as $T\to\infty$, cf. proof of Lemma~\ref{lem:well-posed:pathwise}, we immediately establish the equality in law for $\xi$ and $\xitilde$, thereby concluding the uniqueness of weak solutions.
\end{proof}

\section{Bound on the expected maximum of $F(t)^2$} \label{appendix:F}\label{sec:F} In this section, we will show that under the condition that the autocorrelation $K$ is continuously differentiable, the corresponding stationary process $F(t)$ must satisfy the supremum bound~\eqref{ineq:E.sup.F_t^2<infty}. Thanks to stationarity, it suffices to prove~\eqref{ineq:E.sup.F_t^2<infty} holds for the time interval $[0,T]$, namely, for all $T\ge 0$,
\begin{align*}
\E\sup_{0\le t\le T}F(t)^2<\infty.
\end{align*}  

For convenience, we first recap several notions from the technique of generic chaining in \cite[Chapter 2]{talagrand2014upper}. Consider the time interval $[0,T]$ and the distance $$d(s,t)\overset{def}{=}\sqrt{\E|F(t)-F(s)|^2}.$$ It is well-known that $d$ is a metric in $[0,T]$. For a set $A\subset[0,T]$, we denote by $\triangle(A)$ the diameter of $A$ with respect to metric $d$, that is
\begin{align*}
\triangle(A)\overset{def}{=}\inf_{s,t\in A}d(s,t).
\end{align*}
Next, we provide the definition of an \emph{admissible sequence}.
\begin{definition}
An admissible sequence is an increasing
sequence $\{A_n\}_{n\ge 0}$ of partitions of $[0,T]$ such that $A_0=[0,T]$ and for all $n\ge 1$, card$(A_n)$ is at most $N_n=2^{2^n}$.

Here, increasing sequence means every set of $A_{n+1}$ is contained in some set of $A_n$.
\end{definition}
Given an admissible sequence $A_n$ and a time $t\in[0,T]$, we denote by $A_n(t)$ the element in $A_n$ that contains $t$ and define $\gamma_2(T,d)$ given by
\begin{align*}
\gamma_2(T,d)\overset{def}{=}\inf\sup_{t\in[0,T]}\sum_{n\ge 0}2^{\frac{n}{2}}\triangle(A_n(t)),
\end{align*}
where the infimum is taken over all admissible sequences. We now state the following result asserting that under the conditions imposed on $F(t)$, $\E \sup_{0\le t\le T}F(t)^2$ is always finite.
\begin{lemma} \label{lem:sup.F_t^2}
Let $F(t)$ be a mean-zero Gaussian stationary process whose covariance function $K$ is in $C^1(\rbb)$. Then, for all $T\ge 0$, there exists a positive constant $c(T)$ such that
\begin{equation} \label{ineq:sup.F_t^2}
\E\sup_{0\le t\le T}F(t)^2\le c(T).
\end{equation}
\end{lemma}
\begin{proof}
We first observe that
\begin{align*}
\sup_{0\le t\le T}F(t)^2=\sup_{0\le t\le T}\big(F(t)-F(0)+F(0)\big)^2&\le 2\sup_{0\le t\le T}\big(F(t)-F(0)\big)^2+2F(0)^2 \\
&\le 2\sup_{0\le t,s\le T}\big(F(t)-F(s)\big)^2+2F(0)^2,
\end{align*}
whence
\begin{align*}
\E\sup_{0\le t\le T}F(t)^2\le 2\E\sup_{0\le t,s\le T}\big(F(t)-F(s)\big)^2+2K(0).
\end{align*}
It therefore suffices to establish an upper bound for $\E\sup_{0\le t,s\le T}\big(F(t)-F(s)\big)^2$. 

Now, since $F(t)$ is a mean-zero Gaussian process, $F(t)$ satisfies \cite[inequality (1.4)]{talagrand2014upper}, that is for all $r> 0$
\begin{align} \label{ineq:P(|F_t-F_s|>r)}
\P(|F(s)-F(t)|\ge r)\le  2\exp\Big(-\frac{r^2}{2d(s,t)^2}\Big).
\end{align}
Indeed, by Markov's inequality,
\begin{align*}
\P(|F(s)-F(t)|\ge r)&=\P\Big(\exp\Big(-\frac{r^2}{2|F(s)-F(t)|^2} \Big)\ge \frac{1}{\sqrt{e}}\Big)\le 2\E \exp\Big(-\frac{r^2}{2|F(s)-F(t)|^2} \Big).
\end{align*}
Observe that $f(x)=e^{-r/x}$ is concave down on $(0,\infty)$. So that, Jensen's inequality implies
\begin{align*}
\E \exp\Big(-\frac{r^2}{2|F(s)-F(t)|^2} \Big)
&\le  \exp\Big(-\frac{r^2}{2\E|F(s)-F(t)|^2} \Big),
\end{align*}
which proves~\eqref{ineq:P(|F_t-F_s|>r)}. Now, in light of \cite[inequality (2.49)]{talagrand2014upper}, there exists a positive constant $C$ independent of $T$ such that
\begin{align*}
\E\sup_{0\le t,s\le T}\big(F(t)-F(s)\big)^2\le C\gamma_2(T,d).
\end{align*}
It remains to show that $\gamma_2(T,d)$ is finite. To this end, consider the the following sequence $\{\Atilde_n\}_{n=0}$ given by
\begin{align*}
\Atilde_0=[0,T]\quad\text{and}\quad\Atilde_n=\Big[0,\frac{T}{N_n}\Big)\cup\Big[\frac{T}{N_n},\frac{2T}{N_n}\Big)\dots\Big[\frac{(N_n-1)T}{N_n},T\Big],\quad n\ge 1.
\end{align*}
It is straightforward to check that $\Atilde_n$ is an admissible sequence. For each $t\in[0,T]$, by definition of $\triangle$, we note that
\begin{align*}
\triangle(\Atilde_n(t))=\sup_{s,r\in \Atilde_n(t)}d(s,t)&=\sup_{s,r\in \Atilde_n(t)}\sqrt{\E(F(s)-F(r))^2}\\
&=\sup_{s,r\in \Atilde_n(t)}\sqrt{2(K(0)-K(|s-r|)}.
\end{align*}
By the choice of $\Atilde_n$, for all $r,s\in\Atilde_N(t)$, $|r-s|\le T/N_n$. So that,
\begin{align*}
\sup_{s,r\in \Atilde_n(t)}\sqrt{2(K(0)-K(|s-r|)}=\sup_{0\le s\le T/N_n}\sqrt{2(K(0)-K(s))}.
\end{align*}
Since $K\in C^1(\rbb)$, by the Mean-Value Theorem, for $s\in[0,T/N_n]$
\begin{align*}
|K(0)-K(s)|\le\max_{r\in[0,T]}|K'(r)|\cdot s\le \max_{r\in[0,T]}|K'(r)| \cdot \frac{T}{N_n},  
\end{align*}
implying
\begin{align*}
\triangle(\Atilde_n(t))\le \sqrt{\frac{2T}{N_n}\max_{r\in[0,T]}|K'(r)|}.
\end{align*}
Turning back to $\gamma_2(T,d)$, we note that
\begin{align*}
\gamma_2(T,d)\le \sup_{t\in [0,T]}\sum_{n\ge 0}2^{n/2}\triangle(\Atilde_n(t))&\le\sum_{n\ge 0} 2^{n/2} \sqrt{\frac{2T}{N_n}\max_{r\in[0,T]}|K'(r)|}\\
&=\sqrt{T\max_{r\in[0,T]}|K'(r)|}\sum_{n\ge 0}\frac{2^{\frac{n+1}{2}}}{\sqrt{N_n}}\\
&=\sqrt{T\max_{r\in[0,T]}|K'(r)|}\Big(\sqrt{2}+\sum_{n\ge 1}\frac{2^{\frac{n+1}{2}}}{2^{2^{n-1}}}\Big),
\end{align*}
which is clearly finite. Altogether, we arrive at the bound
\begin{align*}
\E\sup_{0\le t\le T}F(t)^2\le C\sqrt{T\max_{r\in[0,T]}|K'(r)|}+2K(0),
\end{align*}
thereby establishing~\eqref{ineq:sup.F_t^2} and completing the proof.
\end{proof}

\bibliographystyle{abbrv}
\bibliography{gibbsian-gle-bib}

\end{document}